\documentclass[11pt,article,reqno]{amsart}
\usepackage{amsmath, amsthm, amsfonts}

\usepackage[top=25mm,bottom=25mm,left=27mm,right=27mm]{geometry}

\usepackage{color}

\makeatletter
\def\section{\@startsection{section}{1}%
\z@{1\linespacing\@plus\linespacing}{1\linespacing}%
{\bf\centering}}
\def\subsection{\@startsection{subsection}{0}%
\z@{\linespacing\@plus\linespacing}{\linespacing}%
{\bf}}
\def\subsubsection{\@startsection{subsubsection}{0}%
\z@{\linespacing\@plus\linespacing}{\linespacing}%
{\bf}}
\makeatother

\makeatletter
\@addtoreset{equation}{section}
\makeatother

\newtheorem{theorem}{Theorem}[section]

\newtheorem{lemma}[theorem]{Lemma}

\newtheorem{remark}[theorem]{Remark}
\theoremstyle{definition}

\newtheorem{example}[theorem]{Example}

\def\RR{\mathbb{R}}
\def\R{\mathbb{R}}

\def\NN{\mathbb{N}}

\def\PP{\mathbb{P}}
\def\EE{\mathbb{E}}
\def\rP{\mathrm P}
\def\rE{\mathrm E}

\def\cal{\mathcal}
\def\al{\alpha}
\def\be{\beta}
\def\ga{\gamma}
\def\de{\delta}

\def\ep{\varepsilon}

\def\la{\lambda}
\def\sk{\smallskip}
\def\mk{\medskip}

\def\ov{\overline}
\def\un{\underline}
\def\wt{\widetilde}
\def\wh{\widehat}

\def\rd{\mathrm{d}}  % dx in integrals

\newcommand{\bra}[1]{\left\lbrace#1\right\rbrace}

\providecommand{\pro}[1]{\{#1_t, t \geq 0\}}
\providecommand{\proo}[1]{\{#1(x), {x \in \R}\}}
\begin{document}
\title
{On the favorite points of symmetric L\'evy processes}
\author{Bo Li \and Yimin Xiao \and Xiaochuan Yang}

\address{Bo Li. School of Mathematics and Statistics, Central China Normal University\\
Wuhan, 430079, China}
\email{haoyoulibo@hotmail.com}

\address{Yimin Xiao. Dept. Statistics \& Probability, Michigan State University, 48824 East Lansing, MI, USA}
\email{xiaoyimi@stt.msu.edu}

\address{Xiaochuan Yang. Dept. Statistics \& Probability, Michigan State University, 48824 East Lansing, MI, USA}
\email{yangxi43@stt.msu.edu, xiaochuan.j.yang@gmail.com}

\thanks{ 
 \\ 
\noindent 2010 {\it MS Classification}: Primary 60J75, 60J55, 60G15
%Research of ... was supported by ...
}

\begin{abstract}
This paper is concerned with asymptotic behavior (at zero and at infinity) of the favorite points 
of L\'evy processes. By exploring Molchan's idea for deriving lower tail probabilities of Gaussian processes 
with stationary increments, we extend the result of Marcus (2001) on the favorite points to a larger class 
of symmetric L\'evy processes. 
\end{abstract}

\maketitle

\baselineskip 0.5 cm

\bigskip \medskip

\section{Introduction}\label{sec:intro}

Let $X = \{{X_t}, t \ge 0,\, \PP^x\}$ be a real-valued L\'evy process with characteristic exponent $\psi$
which is given by
\begin{align*}
\EE^0[e^{i\lambda X_t}]= e^{-t\psi(\lambda)}.
\end{align*}
%$X$ is said to have a local time
Many sample path properties of $X$, including the existence and regularity of local times,
can be described explicitly in terms of $\psi$. Building upon the seminal results of Kesten
\cite{kesten69}, Hawkes \cite{hawkes86} showed that the local times of $X$, denoted by
$\{L_t^x, t \ge 0, x \in \RR\}$, exist as the Radon-Nikodym derivative of the occupation
measure with respect to the Lebesgue measure on $\RR$, namely, for all $t \ge 0$ and all
Borel measurable functions
$f: \RR \to \RR$, %and
\begin{equation*}
\int_0^t f(X_s)\rd s = \int_\RR f(x)L^x_t\rd x,
\end{equation*}
if and only if  Re$\big(1/(1+\psi) \big)\in L^1(\RR)$. See also Bertoin
\cite[p.126]{Bertoin96book} for more information.
Several authors have studied the a.s. joint continuity of $(t, x) \mapsto L_t^x$ and, finally,
%was studied by  several authors, including Barlow \cite{barlow1985}, Barlow and Hawkes \cite{BarlowHawkes85}. Later
Barlow and Hawkes \cite{BarlowHawkes85}, Barlow \cite{barlow88}, Marcus and Rosen
\cite{MarcusRosen92} established sufficient and necessary conditions for the joint
continuity of the local times of L\'evy processes in terms of their 1-potential
kernel densities.  They also determined the exact uniform and local moduli of
continuity for the local time $L_t^x$ in the space variable $x$. The approaches
in \cite{BarlowHawkes85,barlow88} and \cite{MarcusRosen92}  are different.
\cite{BarlowHawkes85,barlow88} considered general L\'evy processes and their
method was based on Dudley's metric entropy method; while \cite{MarcusRosen92}
considered strongly symmetric Markov processes and made use of  an isomorphism
theorem of Dynkin, which relates sample path behavior of local times of a symmetric
Markov process with those of the associated Gaussian processes. An extension of
the results in \cite{MarcusRosen92} to local times of {non-}symmetric
Markov processes was established by Eisenbaum and Kaspi \cite{EK07}. Recently,
Marcus and Rosen \cite{MR13}  extended the sufficiency part  for the joint
continuity of local times in \cite{MarcusRosen92} to
{Markov} processes which are not necessarily symmetric. Their arguments are
based on an isomorphism theorem of Eisenbaum and Kaspi \cite{EK09} which relates
local times of Markov processes with permanental processes.
We refer to the book of Marcus and Rosen \cite{MarcusRosen06} for a systematic
account on connections between local times of Markov processes and Gaussian processes.

\sk
This paper is concerned with asymptotic properties of the local times
of a class of symmetric L\'evy processes.
Let $X = \{{X_t}, t \ge 0\}$ be a real-valued symmetric L\'evy process.
%with characteristic exponent $\psi$. %Then
%$\psi(\la) \ge 0$ for all $\lambda \in \mathbb R$.
%\textcolor{blue}{Throughout this paper, we will assume that $X$ is recurrent
%and has a jointly continuous local time $\{L_t^x, t \ge 0, x \in \RR\}$.} These are guaranteed respectively by
%our assumptions (C1) and (C2) below; see Remark \ref{Re:C}.
Recall that from \cite{BassEisenbaumShi00} the set of favorite points
(or most visited sites) of $X$ up to time $t$ is defined by
\[
{\cal V}_t = \{ x\in\RR :\, L^x_t=\sup_{y\in\RR} L^y_t\Big\}
\]
and the favorite points process of $X$ is defined by
\begin{align*}
V_t= \inf\Big\{ |x|\in\RR :\, L^x_t=\sup_{y\in\RR} L^y_t\Big\}.
\end{align*}
As pointed out by Bass {\it et al.} \cite{BassEisenbaumShi00} (see also
Marcus \cite{Marcus01}) the choice of $V_t$ is not significant. The
results on $V_t$ mentioned in this paper are still valid if
it is replaced by any other element of ${\cal V}_t$.

{Bass and Griffin \cite{BassGriffin85} initiated the study of favorite points and established a
quite surprising result for Brownian motion. Indeed one might expect that a real-valued Brownian
motion starting from zero admits zero as one of its favorite points in the long run, but Bass and Griffin
showed that the favorite points process of a Brownian motion is actually transient, namely,
$\lim\limits_{t \to \infty} V_t = \infty$. Their proof relies on the Ray-Knight Theorem.} Later, Bass {\it et al.}
\cite{BassEisenbaumShi00} proved that when $X$ is a symmetric stable L\'evy process in $\RR$ of index
$\alpha \in (1, 2)$,  its favorite points process is {still} transient.  {They appealed to a generalized
Ray-Knight Theorem \cite{eisenbaum2000ray} and made a clever use of Slepian's Lemma. }
%Namely, $\lim\limits_{t \to \infty} V_t = \infty$.
Marcus \cite{Marcus01} adopted the approach of \cite{BassEisenbaumShi00} and extended their result to a
large class of symmetric L\'evy processes.  On the other hand, Eisenbaum and Khoshnevisan \cite{EKh02}
studied the hitting probability of the set of favorite  points ${\cal V}_t$. More specifically, they call a compact
set $K \subset \RR$ {\it polar} for ${\cal V} := \bigcup_{t > 0} {\cal V}_t$ if
\[
\PP\big( \exists t > 0: {\cal V}_t \cap K \ne \emptyset \big) = 0.
\]
It has been an open problem to establish a criterion for a general compact set $K$ to be polar. However,
Eisenbaum and Khoshnevisan \cite{EKh02} proved that all singletons are polar for a large class of symmetric
Markov processes, including symmetric stable L\'evy processes in $\RR$ with $\alpha \in (1,2)$.

%the favorite points process of $X$ defined by
%\begin{align*}
% V_t= \inf\Big\{ |x|\in\RR :\, L^x_t=\sup_{y\in\RR} L^y_t\Big\}
%\end{align*}
%as $t \to \infty$ and $t \to 0$, respectively.

%Eisenbaum and Khoshnevisan \cite{EKh02} studied the hitting probability of the set of favorite points.
%In particular, they proved that all singletons are polar for a large class of symmetric Markov processes.
%It is not hard to show that their main result, Theorem 1.3, is applicable to  the symmetric L\'evy processes
%considered in this paper.
 %We refer to \cite{EKh02} for more information on history and ... of studies of the favorite points process
%of random walks and continuous-time Markov processes.

The present paper is mainly motivated by Marcus \cite{Marcus01} and
a remark of Marcus and Rosen \cite[p.527]{MarcusRosen06}. In his study
of the asymptotic behavior of the favorite points process $\{V_t, t \ge 0\}$
as $t \to \infty$, Marcus \cite{Marcus01} assumed that the characteristic
exponent $\psi(\la)$ of a symmetric L\'evy process $X$ in $\RR$ satisfies
the following two conditions:
\begin{enumerate}
\item[(i)] $\psi(\la)$ is regularly varying at zero with index
$1<\al\le 2$ and $1/(1+\psi(\lambda))\in L^1(\RR,d\lambda)$;
%\textcolor{red}{Marcus said on page 2 of his paper that this implies that the local time is jointly continuous,
%but I do not think this is enough.
%One needs $1/\psi(\lambda)$ to decay at infinity at certain rate. X: I agree.}
\item[(ii)] $\sigma_0^2(x)$ is increasing on $[0,\infty)$, where
\begin{align}\label{F:sigma}
\sigma_0^2(x) = \frac{2}{\pi}\int_0^\infty \big(1-\cos(\la x) \big) \frac{\rd\la}{\psi(\la)}, \quad x \in \RR.
\end{align}
\end{enumerate}
We remark first that the regularly varying condition in (i) is somewhat restrictive (e.g.,
as shown by Choi \cite[Proposition 2.4]{choi1994},
the characteristic exponent of a general semi-stable L\'evy process may
not be regularly varying. See Section \ref{sec:example} below for more examples.)
%There are many L\'evy processes whose local times exist and
%whose characteristic exponents are not regularly varying at zero.
Secondly, condition (ii) is not easy to verify even if condition (i) is satisfied.
This issue was also noticed by Marcus and Rosen \cite[p.527]{MarcusRosen06} who
suggested to generalize the change of measure argument due to Molchan
\cite{Molchan99,Molchan00} to obtain a key estimate instead of using Slepian's
Lemma which requires the monotonicity condition (ii).  In this paper, we will relax
the conditions (i) and (ii) significantly and study the asymptotic properties of
$\{V_t, t \ge 0\}$ not only as $t \to \infty$, but also as $t \to 0$.  Moreover,
our results also show that Theorem 1.3 and Proposition 1.4 of Eisenbaum and
Khoshnevisan \cite{EKh02} are applicable to L\'evy processes that satisfy Condition
(C1). Hence all singletons are polar for the set of favorite points of such a L\'evy
process.

%Moreover, we consider a general class of Gaussian processes with stationary increments whose variograms
%are not necessarily regularly varying at zero and establish their uniform moduli of continuity. In particular,
%we allow the variogram to be slowly varying at zero and our theorem is applicable to the Gaussian processes
% in Mocioalca and Viens \cite{MocioalcaViens07}.

%Our results on uniform modulus of continuity are new. By combining them with Dynkin's  isomorphism theorem
%\cite{MarcusRosen06}, we derive exact modulus of continuity for the local times of a class of symmetric L\'evy
%processes, which is not covered by Theorem 9.5.15 of Marcus and Rosen \cite{MarcusRosen06}.

%of $X$ defined by
%\begin{align*}
% V_t= \inf\Big\{ |x|\in\RR :\, L^x_t=\sup_{y\in\RR} L^y_t\Big\}
%\end{align*}
%as $t \to \infty$ and $t \to 0$, respectively. %round zero and infinity for a large class of symmetric L\'evy processes.
%To this end,  a transform of $\psi$ plays a key role. Define $\sigma_0: \RR\to [0,\infty)$ by
%\begin{align*}
%\sigma_0^2(x) = \frac{2}{\pi}\int_0^\infty (1-\cos(\la x)) \Delta(\rd\la),
%\end{align*}
%where $\Delta(\rd \la) = \rd\la/\psi(\la)$, which will be customarily called the spectral measure
%of the associate Gaussian process.

%Recall that the potential kernel for the recurrent process $X$ is defined as
%\begin{align*}
%\sigma_0^2(x) = \frac{2}{\pi}\int_0^\infty (1-\cos(\la x)) \frac{\rd\la}{\psi(\la)}, \quad x \in \RR.
%\end{align*}

Now we introduce the conditions used in this paper. Recall that for a symmetric
L\'evy process $X$ in $\RR$, its characteristic exponent $\psi$ is nonnegative
and has the following L\'evy-Khintchine representation:
\begin{align}\label{eq:LK_formula}
\psi(\la) = A^2 \lambda^2 + \int_0^\infty \big(1-\cos(x\la) \big) \nu(\rd x),
\end{align}
where $\nu$ is a Borel measure on $(0, \infty)$ that satisfies
$\int_0^\infty (1 \wedge x^2) \nu(\rd x) < \infty$ and is called the L\'evy
measure of $X$. Except in Section 6.1, we tacitly assume that the following holds:
\begin{equation}\label{Con:general}
\int_1^\infty \frac{d\lambda}{\psi(\lambda)} < \infty \ \ \hbox{ and }\ \ \
\int_0^1 \frac{d\lambda} {\psi(\lambda)} = \infty.
\end{equation}
The convergence of the first integral in (\ref{Con:general}) is equivalent to
the existence of local times  \cite{hawkes86}, while the second condition in
(\ref{Con:general}) is equivalent to the recurrence of $X$ (see Port and Stone
\cite[Thm 16.2]{PortStone71AIF}).

As in Bass {\it et al.} \cite{BassEisenbaumShi00} and Marcus \cite{Marcus01},
we will focus on  symmetric, pure-jump L\'evy processes, i.e., we assume that
$A = 0$. Some results for symmetric L\'evy processes with a Brownian motion
part can be derived from those for the pure jump part, see Section \ref{sec_ext}.

We will further assume that the L\'evy measure $\nu$ is absolutely continuous
$\nu(\rd x) = \rd x/\theta(x)$. It is possible to consider the case when $\nu$
is discrete or singular using the methods in Berman \cite{Berman87} or Luan and
Xiao \cite{LuanX12},  respectively. Since this will increase the size of the
paper and most of the deviation is of technical nature, we do not carry it out here.

Our conditions are on the asymptotic behavior of the L\'evy measure $\nu$.
\begin{itemize}
\item[(C1)] There exist constants $\un\al$ and $\ov\al$ such that
\begin{align*}
1<\un\al= \liminf_{x\to 0} \frac{x/\theta(x)}{\nu(z: |z|\ge x)} \le  \limsup_{x\to 0} \frac{x/\theta(x)}{\nu(z: |z|\ge x)}=\ov\al <2.
\end{align*}
%\item[(C2)]  \blue{condition on $\psi$} (via Pitman, equivalent to: For $x$ sufficiently small,  $\sigma_0^2(x)
%\asymp (\log 1/|x|)^{-\ga}$ with $\ga>1$.) \blue{and} There exists a positive constant $\de$ such that $\sigma_0^2(x)$
%is concave in the interval $[0,\de]$.
\item[(C2)] There exist constants $\un\be$ and $\ov\be$ such that
\begin{align*}
1<\un\be= \liminf_{x\to \infty} \frac{x/\theta(x)}{\nu(z: |z|\ge x)} \le  \limsup_{x\to \infty} \frac{x/\theta(x)}{\nu(z: |z|\ge x)}=\ov\be <2.
\end{align*}
%\item[(C4)] There exist $\al$ and $\be$ such that the L\'evy measure satisfies
%\begin{align*}
%< \al=  \liminf_{\la\to 0} \frac{\la/\theta(\la)}{\nu(\xi: |\xi|\ge \la)} \le  \limsup_{\la\to 0} \frac{\la/\theta(\la)}{\nu(\xi: |\xi|\ge \la)}=\be \le 2.
%\end{align*}
\end{itemize}

\begin{remark} \label{Re:C}
{\rm Some comments about Conditions (C1) and (C2) are in order. }
\begin{itemize}
\item Berman \cite{Berman87} used (C2) type conditions to prove
local nondeterminism property for a centered Gaussian process
$Y = \{Y(t), t \in \RR\}$ with stationary increments whose variance
of increment $\EE\big[(Y(t + \lambda) - Y(t))^2\big]$ can be
represented (up to a constant) by (\ref{eq:LK_formula}) with $A = 0$.
Xiao \cite{Xiao07} extended Berman's argument to the random field
setting and proved that a similar condition implies strong local
{nondeterminism} for Gaussian random fields. Here we show that
Conditions (C1) and (C2) imply asymptotic properties of the characteristic
exponent $\psi(\la)$ of $X$ as $\lambda \to \infty$ and $\lambda \to 0$;
see Lemma \ref{lem:on_conditions} and \ref{lem:on_conditions_3} below.

%Roughly speaking, this condition implies that $\psi(\la)$ is bounded from both sides by power functions with possibly different order,
\item It follows from Lemma \ref{lem:on_conditions2} that, under Condition
(C1), the function $\sigma_0^2(x)$ decays at a power rate as $x \to 0$. Then,
by applying the criterion of Barlow and Hawkes \cite{BarlowHawkes85} or
Marcus and Rosen \cite{MarcusRosen92}, one can show that $X$ has a jointly
continuous local time.

\item It follows from Condition (C2)  and Lemma \ref{lem:on_conditions_3}
that  $\int_0^1 \frac{d\lambda} {\psi(\lambda)} = \infty$.
Hence $X$ is recurrent thanks to the Spitzer-type criterion (see \cite[Thm 16.2]{PortStone71AIF} or \cite[p.33]{Bertoin96book}).
By \cite[p.140]{MarcusRosen06}, the $0$-potential density of $X$  satisfies
$$u(0,0)=\int_0^\infty \frac{\rd \la}{\psi(\la)}=\infty,$$
and $\PP^x(T_0<\infty)=1$ for all $x\in\RR$, where $T_0 = \inf\{t> 0: X_t = 0\}$
is the first hitting time of zero by $X$.
%\item In fact, from the proofs below we may see that the main results of this paper hold as long as the conclusion
%of Lemma \ref{lem:on_conditions} or Lemma \ref{lem:on_conditions_3} holds, which are conditions on
%the characteristic exponent $\psi$. In Section \ref{sec:example}, we present an example where (C1) (resp. (C2)) fails,
%but the conclusion of Lemma \ref{lem:on_conditions} (resp. Lemma \ref{lem:on_conditions_3}) holds.
\end{itemize}
\end{remark}
%Notice that the index $\be_1$, $\be_2$ in (C3) might take critical value, but not at the same time.
%We will prove that (C3) implies (C1).  Condition (C2) is  used only when $\sigma_0^2(x)$ is slowly varying.
%It is plain that the concavity holds when we have $\sigma_0^2(x)=(\log 1/|x|)^{-\ga}$. One needs the
%exact asymptotic behavior of $\sigma_0^2(x)$ to prove an upper bound for the uniform moduli of
%Gaussian processes, whereas the concavity is enough in order to derive a lower bound.

\mk
{Let us explain the novelty of our techniques. } The function $\sigma_0^2(x)$
in (\ref{F:sigma}) plays an important role in characterizing the joint continuity and
other analytic properties of the  local times of $X$.  Set the tail function %of the spectral measure
\begin{align*}
\phi(x)&=2\int_{1/x}^\infty {\rd \la\over \psi(\la)}, \quad x\ge 0.
\end{align*}
It is clear that  $\phi(x)$ is nondecreasing in $x$. It follows from Lemma 2.3
and Theorem 2.5 in \cite{Xiao07} that, under Condition (C1), the functions
$\phi(x)$ and  $\sigma_0^2(x)$ are comparable as $x \to 0$.  In Section 2,
we establish further connections between the asymptotic  behaviors of
$\sigma_0^2(x)$ with those of $\psi(\la)$.  These allow us to remove the
monotonicity assumption (ii) on $\sigma_0^2(x)$ in Marcus \cite{Marcus01}.
%Then, by applying the criterion Barlow and Hawkes \cite{BarlowHawkes85} and Marcus and Rosen \cite{MarcusRosen92}, one
%can show that $X$ has a jointly continuous local time.
{Following \cite{BassEisenbaumShi00}, through the generalized Ray-Knight
Theorem we transfer some important estimates for the local times to certain
estimates of the associated Gaussian processes. In \cite{BassEisenbaumShi00,Marcus01},
the Gaussian estimates were obtained by Slepian's Lemma. Here we obtain these
estimates by applying the Cameron-Martin change of measure formula, together
with a uniform upper bound for the maximum location of Gaussian processes with
stationary increments.  }

The following are the main results of this paper. Part (ii) of Theorem \ref{theo}
is an extension of Theorem 1.1 of Marcus \cite{Marcus01} and Theorem \ref{Th:Polar}
is an extension of Theorem 5.2 of Eisenbaum and Khoshnevisan \cite{EKh02}.

\begin{theorem}\label{theo}
Let $X$ be a symmetric L\'evy process with characteristic exponent $\psi$ that
satisfies \eqref{Con:general}.
\begin{itemize}
\item[(i)] Assume that (C1) holds. For all $a>2(\ov\al-1)/(2-\ov\al)$,
\begin{align*}
\lim_{t\to 0}  \frac{ V_t}{\phi^{-1}\left( \frac{L^0_t}{(\log L^0_t)^a}\right)}
=\infty  \quad \PP^0 \mbox{-a.s.}
\end{align*}

\item[(ii)] Assume that (C2) holds. For all $a>2(\ov\be-1)/(2-\ov\be)$,
\begin{align*}
\lim_{t\to \infty}  \frac{ V_t}{\phi^{-1}\left( \frac{L^0_t}{(\log L^0_t)^a}\right)}
=\infty \quad \PP^0 \mbox{-a.s.}
\end{align*}

\end{itemize}
\end{theorem}

{
\begin{remark}
Our conditions impose certain asymptotic behavior of the L\'evy measure,
which in turn implies certain asymptotic behavior of the characteristic
exponent and a ratio control for the tail function of the L\'evy measure,
see \eqref{eq:ratio_contr_1}-\eqref{eq:equiv_1} and \eqref{eq:ratio_contr_3}-\eqref{eq:equiv_3}
in Section 2.  In fact, from the proofs below we may see that Part (i) of
Theorem \ref{theo} (resp. Part (ii)) is valid as long as
\eqref{eq:ratio_contr_1}-\eqref{eq:equiv_1} (resp. \eqref{eq:ratio_contr_3}-\eqref{eq:equiv_3})
hold. Both sets of conditions are useful since L\'evy processes encountered
in the literature may be specified explicitly either by the L\'evy measure
or by the characteristic exponent. From this and Proposition 2.4 of
Choi \cite{choi1994}, it follows that Theorem \ref{theo} holds for any
symmetric semi-stable L\'evy process in $\RR$ with index $\alpha \in (1, 2]$.
In Section \ref{sec:example}, we also present an example where (C1) fails,
but \eqref{eq:ratio_contr_1}-\eqref{eq:equiv_1}
hold, thus Part (i) of Theorem \ref{theo} continues to hold.
\end{remark}}

%for more examples of L\'evy processes
%which do not satisfy the conditions in Marcus \cite{Marcus01}, but
%our Theorem \ref{theo} is applicable.

\begin{theorem}\label{Th:Polar}
If Conditions \eqref{Con:general} and (C1) hold, % with $\un\al - \frac{\ov\al}{2}>\frac 1 2$,
then all singletons are polar for the set ${\cal V}$ of favorite points.
Namely, for every $x \in \RR$,
\[
\PP^x\big(\exists t > 0 \hbox{ such that }  x \in {\cal V}_t   \big) = 0.
\]
\end{theorem}

The rest of the paper is organized as follows. Some preliminary results are
presented in Section \ref{sec:2}. We obtain upper and lower tail
estimates  for the maxima of Gaussian processes in Section \ref{sec:tail_estimates}.
Using these estimates, we prove Theorems \ref{theo} and \ref{Th:Polar} in
Section \ref{sec:proof_theo}.  Some examples are given in Section \ref{sec:example}.
Possible extensions are discussed in Section \ref{sec_ext}.

We end the Introduction with some notations. We use $\PP^x$ and $\EE^x$ to
denote the law and the  expectation of any L\'evy process with starting point
$x\in\RR$. For simplicity, we write $\PP$ and $\EE$ when $x=0$. In Sections 3
and 4, the law and the expectation of any auxiliary Gaussian process are
denoted by $\rP$ and $\rE$. We use $c,C$ to denote generic constants whose
value may change from line to line.

\section{Preliminaries}\label{sec:2}

\subsection{General facts}
Introduce the tail function of the L\'evy measure $\nu$:
\begin{align*}
\pi(x)&=2\int_{1/x}^\infty {\rd z\over \theta(z)}, \qquad \forall x > 0.
\end{align*}
The function $\pi(x)$  is important because, under our condition, it is equivalent to $\psi(\la)$ around infinity and has the
advantage of being a monotone function. More precisely, the following lemma which is a reminiscence of
\cite[Lem 2.3 and Thm 2.5]{Xiao07} gathers useful facts derived from (C1).

\begin{lemma}\label{lem:on_conditions}
 Assume that (C1) holds. For any $\ep\in(0, \min(\un\al-1, 2-\ov\al))$, there exists a  finite positive constant $K_0$ such that for all $y>x>K_0$,
 \begin{align}
 \left(x\over y \right)^{\ov\al +\ep} &\le  { \pi(x)\over \pi(y) } \le  \left( x\over y\right)^{\un\al-\ep}.\label{eq:ratio_contr_1}
 \end{align}
Moreover,
\begin{align}
0<\liminf_{\la\to \infty} \frac{\psi(\la)}{\pi(\la)} \le \limsup_{\la\to \infty} \frac{\psi(\la)}{\pi(\la)} <\infty.\label{eq:equiv_1}
\end{align}
\end{lemma}
\begin{proof}\,
The proof which led to  \cite[Lem. 2.3]{Xiao07}  yields plainly \eqref{eq:ratio_contr_1} by considering $x,y$ around infinity.
Using similar arguments that led to \cite[Thm. 2.5]{Xiao07}, but for large $\la$ rather than for small ones, entails \eqref{eq:equiv_1}.
Let us prove \eqref{eq:equiv_1}  for the sake of completeness. By \eqref{eq:LK_formula}, we can write $\psi(\la)$ as
\begin{align*}
\psi(\la) = \int_0^{1/\la} (1-\cos(\la x))\frac{\rd x}{\theta(x)} + \int_{1/\la}^\infty (1-\cos(\la x))\frac{\rd x}{\theta(x)} := I_1+I_2.
\end{align*}
By (C1), there exists $0<r_0<1/K_0$ such that for any $x<r_0$,
\begin{align*}
\un\al -\ep \le  \frac{x/\theta(x)}{\pi(1/x)} \le \ov\al+\ep.
\end{align*}
So for $\la>K_0$,  it follows from \eqref{eq:ratio_contr_1} that
\begin{align*}
\frac{I_1}{\pi(\la)} &\ge (\un\al-\ep)\int_0^{1/\la} (1-\cos(\la x))\frac{\pi(1/x)}{\pi(\la)} \frac{\rd x}{x} \\
&\ge  (\un\al-\ep)\int_0^{1/\la} (1-\cos(\la x))\left(\frac{1}{x\la}\right)^{\un\al-\ep} \frac{\rd x}{x} \\
&= (\un\al-\ep)\int_0^1 (1-\cos x) \frac{\rd x}{x^{\un\al+1-\ep}} =c>0.
\end{align*}
This proves the left inequality in \eqref{eq:equiv_1}.

To prove the right inequality in \eqref{eq:equiv_1}, we control both $I_1$ and $I_2$.  Firstly, for $\la>K_0$,
we still use \eqref{eq:ratio_contr_1} and derive
\begin{align*}
\frac{I_1}{\pi(\la)} &\le (\ov\al+\ep)\int_0^{1/\la} (1-\cos(\la x))\frac{\pi(1/x)}{\pi(\la)} \frac{\rd x}{x} \\
&\le  (\ov\al+\ep)\int_0^{1/\la} (1-\cos(\la x))\left(\frac{1}{x\la}\right)^{\ov\al+\ep} \frac{\rd x}{x} \\
&= (\ov\al+\ep)\int_0^1 (1-\cos x) \frac{\rd x}{x^{\ov\al +1+\ep}} =c <\infty,
\end{align*}
where we have used the fact $\ov \al + \ep<2$.   Next, bounding from above the integrand by $1$,
we obtain $I_2/\pi(\la)\le 1$ for all $\la$.  Combining the two bounds completes the proof.  $\square$
\end{proof}

%similar properties forIt was shown in \cite[Lem 2.3 and Thm 2.5]{Xiao07} that, under Condition (C1),  the functions $\phi(x)$ and  $\sigma_0^2(x)$ are comparable

Next we show that the ratio control \eqref{eq:ratio_contr_1} and the equivalence \eqref{eq:equiv_1}  on $\psi$ as $\la \to \infty$
imply similar properties of  $\phi(x)$ and the maxima of $\sigma^2_0(x)$ as $x \to 0$. The tool is \cite[Prop. 1]{KaletaSztonyk15},
see also  \cite{Schilling98JTP}.

\begin{lemma}[{\cite[Prop. 1]{KaletaSztonyk15}}]\label{lem:KS} Let
\begin{align*}
H(x)=  \int 1\wedge \left( \la\over x\right)^2 \frac{\rd \la}{\psi(\la)}
\end{align*}
and $\wh\sigma^2_0(h)=\max_{|x|\le h}\sigma^2_0(x)$.  There exists a constant $c>0$ such that for all $x>0$,
\begin{align*}
c\int_0^{1/x} (\la x)^2 \frac{\rd\la}{\psi(\la)} \le\sigma_0^2(x) \le 2H(1/x), \\
c H(1/x)\le \wh\sigma_0^2(x) \le 2H(1/x).
\end{align*}
\end{lemma}

\begin{lemma}\label{lem:on_conditions2} Assume that (C1) holds.  There are positive finite constants $\al_1<\al_2$ such that for
any $0<\ep<\al_1$, there exists $r_0>0$ such that for $0<x<y<r_0$,
\begin{align}
 \left(x\over y \right)^{\al_2 +\ep} &\le  { \phi(x)\over \phi(y) } \le  \left( x\over y\right)^{\al_1-\ep}.\label{eq:ratio_contr_2}
 \end{align}
 Further,  one has
 \begin{align}
0<\liminf_{x\to 0} \frac{\wh\sigma^2_0(x)}{\phi(x)} \le \limsup_{x\to 0} \frac{\wh\sigma^2_0(x)}{\phi(x)} <\infty.\label{eq:equiv_2}
 \end{align}
\end{lemma}

\begin{proof}\,
A direct application of \cite[Lem. 2.3]{Xiao07} yields \eqref{eq:ratio_contr_2} as long as we show
\begin{align}\label{eq_in_lem2.3}
0<\al_1=\liminf_{\la\to\infty} \frac{\la/\psi(\la)}{\phi(1/\la)}\le \limsup_{\la\to \infty}\frac{\la/\psi(\la)}{\phi(1/\la)} =\al_2<\infty.
\end{align}
Let us start with the left inequality in \eqref{eq_in_lem2.3}.  By Lemma \ref{lem:on_conditions}, there exist constants $0<c,\,
C<\infty$ such that for all $\la>C$,
\begin{align*}
 \frac{\la/\psi(\la)}{\phi(1/\la)}= \frac{\la}{2\int_{\la}^\infty \frac{\psi(\la)}{\psi(z)}\rd z} \ge \frac{c\la}{\int_{\la}^\infty \frac{\pi(\la)}{\pi(z)}\rd z}
 \ge \frac{c\la}{\int_{\la}^\infty (\la/z)^{\un\al-\ep}\rd z}:=c_0>0.
\end{align*}
Similarly, there exists $c'<\infty$ such that for $\la>C$,
\begin{align*}
 \frac{\la/\psi(\la)}{\phi(1/\la)}\le \frac{c'\la}{\int_{\la}^\infty \frac{\pi(\la)}{\pi(z)}\rd z} \le \frac{c'\la}{\int_{\la}^\infty (\la/z)^{\ov\al+\ep}\rd z}:=c'_0<\infty,
\end{align*}
which implies  the desired right inequality in \eqref{eq_in_lem2.3}.

\sk
Next we show \eqref{eq:equiv_2}.  Thanks to Lemma \ref{lem:KS}, it suffices to show  \eqref{eq:equiv_2} with
$\wh\sigma^2_0(x)$ replaced by $H(1/x)$.   Observe that
\begin{align}\label{eq:H}
H\left({1\over x}\right)= 2\int_{0}^{1/x} (\la x)^2 \frac{\rd \la}{\psi(\la)} + 2\int_{1/x}^\infty \frac{\rd \la}{\psi(\la)}.
\end{align}
Thus, the left inequality in \eqref{eq:equiv_2} holds trivially.  On the other hand,  for $x<1/C$,
\begin{align*}
\int_{0}^{1/x} (\la x)^2 \frac{\rd \la}{\psi(\la)}&= x^2 \left(\int_0^C \la^2 \frac{\rd \la}{\psi(\la)} + \int_C^{1/x} \la^2 \frac{\rd \la}{\psi(\la)} \right).
\end{align*}
The first integral  in the parenthesis is finite and the second integral goes to infinity as $x\to 0$ (\cite[Lem.4.2.2]{MarcusRosen06}). Therefore,
\begin{align*}
\int_{0}^{1/x} (\la x)^2 \frac{\rd \la}{\psi(\la)}\le cx^2\int_C^{1/x} \la^2 \frac{\rd \la}{\psi(\la)} =\frac{cx^2}{\psi(1/x)} \int_C^{1/x} \frac{\psi(1/x)}{\psi(\la)} \la^2\rd \la
\end{align*}
which by Lemma \ref{lem:on_conditions} is bounded from above by
\begin{align}\label{eq_in_lem2.3_2}
 \frac{cx^2}{\psi(1/x)} \int_C^{1/x} \frac{\pi(1/x)}{\pi(\la)}  \la^2\rd \la &\le  \frac{cx^2}{\psi(1/x)} \int_C^{1/x} \left( 1\over x\la\right)^{\ov\al+\ep} \la^2\rd \la \nonumber \\
 &\le \frac{c}{\psi(1/x)x}\int_0^1 \left( 1\over \la\right)^{\ov\al+\ep-2}\rd \la.
\end{align}
We can assemble the last integral in \eqref{eq_in_lem2.3_2}  into the constant $c$ since $\ov\al>1$ and $\ep$ is small.
To finish the proof of the right inequality in \eqref{eq:equiv_2}, it remains to show that
\begin{align}\label{eq3_lem2.3}
\frac{1}{x\psi(1/x)}\le c\int_{1/x}^\infty \frac{\rd \la}{\psi(\la)}.
\end{align}
for all $x >0$ small. Indeed,  combining \eqref{eq:H}, \eqref{eq_in_lem2.3_2} and \eqref{eq3_lem2.3}, together with Lemma \ref{lem:KS},
shows that $\wh\sigma^2_0(x)\le c\phi(x)$ for all $x>0$ sufficiently small, as desired.  By Lemma \ref{lem:on_conditions},
for all $0< x<1/C$,
\begin{align}\label{eq:p.7}
\int_{1/x}^\infty \frac{\psi(1/x)}{\psi(\la)}\rd \la \ge  c \int_{1/x}^\infty \left( 1\over \la x\right)^{\ov\al+\ep}\rd \la = \frac{c}{x} \int_1^\infty \frac{\rd\la}{\la^{\ov\al+\ep}}
\end{align}
from which \eqref{eq3_lem2.3} follows.   The proof is now complete.
$\square$
\end{proof}

\begin{remark}\label{rem:2.4} {\rm
In fact, the two terms in \eqref{eq:H} are of the same order under Condition (C1). To see this, it remains to prove that
there exists $c>0$ such that
$$x^2\int_C^{1/x}\la^2 \frac{\rd\la}{\psi(\la)}\ge c\int_{1/x}^\infty \frac{\rd\la}{\psi(\la)} =c \phi(x). $$
This is true since one can reverse inequalities in \eqref{eq_in_lem2.3_2} and \eqref{eq:p.7} if we replace all
$\ov\al+\ep$ by $\un\al-\ep$. Consequently, $\sigma_0^2(x)\asymp \phi(x)$ as $|x|\to 0$ by the first inequality in Lemma \ref{lem:KS}.
}
\end{remark}

Interchanging the roles of zero and infinity, one deduces similar facts under (C2).
\begin{lemma}\label{lem:on_conditions_3}
 Assume that (C2) holds. For any $\ep\in(0, \min(\un\be-1, 2-\ov\be))$, there exists a  finite positive constant $r_1$ such that for all $0<x<y<r_1$,
 \begin{align}
 \left(x\over y \right)^{\ov\be +\ep} &\le  { \pi(x)\over \pi(y) } \le  \left( x\over y\right)^{\un\be-\ep}.\label{eq:ratio_contr_3}
 \end{align}
Moreover,
\begin{align}
0<\liminf_{\la\to 0} \frac{\psi(\la)}{\pi(\la)} \le \limsup_{\la\to 0} \frac{\psi(\la)}{\pi(\la)} <\infty.\label{eq:equiv_3}
\end{align}
\end{lemma}

\begin{lemma}\label{lem:on_conditions4} Assume that (C2) holds.  There are positive finite constants $\be_1<\be_2$ such
that for any $0<\ep<\be_1$, there exists $K_1>0$ such that for all $y>x>K_1$,
\begin{align}
 \left(x\over y \right)^{\be_2 +\ep} &\le  { \phi(x)\over \phi(y) } \le  \left( x\over y\right)^{\be_1-\ep}.\label{eq:ratio_contr_4}
 \end{align}
 Further,  one has
 \begin{align}
0<\liminf_{x\to \infty} \frac{\wh\sigma_0^2(x)}{\phi(x)} \le \limsup_{x\to \infty} \frac{\wh\sigma^2_0(x)}{\phi(x)} <\infty.\label{eq:equiv_4}
 \end{align}
\end{lemma}

\noindent{\it Proof of Lemmas \ref{lem:on_conditions_3} and  \ref{lem:on_conditions4}}\,
The ratio control \eqref{eq:ratio_contr_3} is exactly \cite[Lem. 2.3]{Xiao07}, and the equivalence \eqref{eq:equiv_3}
is \cite[Thm. 2.5]{Xiao07}.
To prove \eqref{eq:ratio_contr_4}, it suffices to show that
\begin{align*}%\label{eq:lem2.5_1}
0<\be_1=\liminf_{\la\to 0} \frac{\la/\psi(\la)}{\phi(1/\la)}\le \limsup_{\la\to 0}\frac{\la/\psi(\la)}{\phi(1/\la)} =\be_2<\infty.
\end{align*}
This can be handled in exactly the same way as \eqref{eq_in_lem2.3} is proved, with an application of Lemma
\ref{lem:on_conditions_3} instead of Lemma \ref{lem:on_conditions}. Let us prove \eqref{eq:equiv_4}. Since there
is no obvious way to control the value of $\be_2$, one cannot apply directly Lemma \ref{lem:on_conditions} to derive
\eqref{eq:equiv_4}. But Lemma \ref{lem:KS} which holds for all $x>0$ is still applicable. Observe that the left
inequality of \eqref{eq:equiv_4} follows from \eqref{eq:H}. On the other hand,   for $x>1/r_1$,
\begin{align*}
\int_{0}^{1/x} (\la x)^2 \frac{\rd \la}{\psi(\la)} &= \frac{1}{x\psi(1/x)}\int_0^1\la^2 \frac{\psi(1/x)}{\psi(\la/x)}\rd \la \\
&\le \frac{c}{x\psi(1/x)}\int_0^1\la^2\left(1\over \la\right)^{\ov\be+\ep}\rd\la = \frac{c}{x\psi(1/x)},
\end{align*}
where we have used Lemma \ref{lem:on_conditions_3} in the second inequality and  $\ov\be<2$. To finish the proof,
it remains to show \eqref{eq3_lem2.3} for all large $x$. For $x>2/r_1$, write
\begin{align*}
\int_{1/x}^\infty \frac{\rd \la}{\psi(\la)} = \int_{1/x}^{r_1}\frac{\rd \la}{\psi(\la)} + \int_{r_1}^\infty \frac{\rd \la}{\psi(\la)}:=J_1+J_2.
\end{align*}
Observe that $J_2<\infty$ by assumption (which is the necessary and sufficient condition for the existence of local times).
%Also, by Lemma \ref{lem:on_conditions_3} for all $x>1/r_1$,
%\begin{align*}
%\frac{1}{x\psi(1/x)} \ge \frac{(r_1 x)^{\un\be-\ep}}{\psi(r_1)x}.
%\end{align*}
%This implies that $\lim_{x\to\infty}1/(x\psi(1/x))=\infty$ as $\un\be>1$. Hence, to show \eqref{eq3_lem2.3}, we need to
%prove that $J_1\ge c/(x\psi(1/x))$.
Using again Lemma \ref{lem:on_conditions_3}, one has for $x>2/r_1$
\begin{align*}
J_1 &= \frac{1}{\psi(1/x)} \int_{1/x}^{r_1} \frac{\psi(1/x)}{\psi(\la)}\rd\la  \\
 &\ge  \frac{c}{\psi(1/x)}\int_{1/x}^{r_1} \left( 1\over x\la\right)^{\ov\be+\ep}\rd\la \ge \frac{c}{x\psi(1/x)}
 \int_1^{2}\left( 1\over \la\right)^{\ov\be+\ep}\rd\la =  \frac{c}{x\psi(1/x)},
\end{align*}
as desired. This finishes the proof\maketitle.
$\square$
%\end{proof}

\subsection{An isomorphism theorem}

We recall a generalized second Ray-Knight Theorem %Dynkin-type isomorphism theorem
due to Eisenbaum {\it et al.} \cite{eisenbaum2000ray};
see also Marcus and Rosen \cite{MarcusRosen06}.
Let $X= \pro X$ be a strongly symmetric Borel right process with values in
$\RR$ with continuous $\al$-potential densities $u^\al(x,y)$. Then the local
times $\bra{L^x_t, t\ge 0, x\in \RR}$ exist and satisfy
\begin{equation}
\EE^x\left[\int_0^{\infty} e^{-\al t}\rd L^y_t\right] = u^\al(x,y).
\end{equation}
Denote $u_{T_0}(x,y) = \EE^x[L^y_{T_0}]$ which is the $0$-potential density of $X$ killed
at the first time it hits zero. Set $\tau(t)=\inf\{s\ge 0: L^0_s> t\}$.
%where $T_0$ is the first hitting time of zero by $X$.
We state the generalized second Ray-Knight Theorem in the recurrent case from
\cite[Thm 1.1]{eisenbaum2000ray} or \cite[Thm 8.2.2]{MarcusRosen06}.

\begin{theorem}\label{RK}
Assume that the 0-potential density of $X$ satisfies $u(0,0)=\infty$ and $\PP^x(T_0<\infty)>0$
for all $x$.
% \footnote{The first half of the condition is equivalent to recurrence when LT exists. The second half of the condition is not mentioned in \cite{eisenbaum2000ray}, but mentioned in \cite{MarcusRosen06}. I think they should have mentioned it in \cite{eisenbaum2000ray} when considering general  symmetric Markov processes. In the RHS of second equality of Formula (5.3) on p.1790 of \cite{eisenbaum2000ray}, there should be some $\PP^x(T_0<\infty)$ when applying the strong Markov property. Here, the condition is satisfied under  (C2) by Remark 1.1.}
Let $\eta= \{\eta_x, x \in \RR\}$ be a mean-zero Gaussian process on a probability space $(\Omega_\eta, {\cal F}_\eta, \rP_\eta)$
with covariance function $u_{T_0}(x,y)$. Then for any $t>0$ and any countable subset $D \subset \RR$, under
$\PP\times \rP_\eta$, in law
\begin{equation}
\bra{L^x_{\tau(t)}+{\eta_x^2 \over 2}; x\in D}=\bra{{1\over 2}{(\eta_x+\sqrt{2t})^2}; x\in D}.
\end{equation}
\end{theorem}

As we have mentioned earlier,  we will write the probability $\rP_\eta$ as $\rP$ and its expectation as $\rE$.

It follows from Theorem 6.1 of Eisenbaum {\it et al.} \cite{eisenbaum2000ray} that, for a recurrent symmetric
L\'evy process with characteristic exponent $\psi(\la)$, the associated Gaussian process $\eta$ is
centered with stationary increments such that its covariance  $u_{T_0}(x,y)$ is given by
\begin{align}\label{Def:u}
u_{T_0}(x,y) =\frac{1}{2} (\sigma_0^2(x) + \sigma^2_0(y) -\sigma^2_0(x-y)).
\end{align}
or equivalently it has variogram $ \rE[(\eta(x)-\eta(y))^2]=\sigma_0^2(x-y)$. Here, $\sigma_0^2(x)$ is defined in
\eqref{F:sigma}.

The following is the Cameron-Martin (change of measure) formula for Gaussian processes.

\begin{theorem}[{\cite[Thm 11.4.1]{MarcusRosen06}}]\label{CM}
Let $\{G_t; t\in S\}$ be a mean zero Gaussian process with continuous covariance $\Gamma(x,y)$ and let $H(\Gamma)$
be the reproducing kernel Hilbert space generated by $\Gamma$. Let $f\in H(\Gamma)$. Then for any measurable functional $F$,
\begin{align*}
\rE[F(G_.+f(\cdot))] = e^{-\|f\|^2/2}\rE[F(G_.)e^{G(f)}],
\end{align*}
where $\|\cdot\|$ is the norm on $H(\Gamma)$ and $G(f)$ is a Gaussian random variable with mean zero and variance $\|f\|^2$.
\end{theorem}

For the  Gaussian process associated to a L\'evy process with characteristic exponent $\psi(\la)$ in Theorem \ref{RK}, the reproducing kernel
Hilbert space is endowed with the norm
\begin{align*}
\|f - f(0)\|^2 = \frac{1}{2\pi}\int_0^\infty |\psi(\la)| |\hat f(\la)|^2 \rd \la,
\end{align*}
where $\hat f$ is the Fourier transform of $f$, see \cite[Sec 11.7]{MarcusRosen06}.

\subsection{Location of the leftmost maximum}
The next lemma is adapted from Theorem 3.1(b) of Samorodnitsky and Shen \cite{SamorodnitskyShen13}, see also \cite{Shen16SPA}.

\begin{lemma}\label{lem:loc_max} Let $\eta$ be a continuous process with stationary increments. Define $\tau_{[a,b]}$ to
be the leftmost maximum location for $\eta$ on the interval $[a,b]$, that is
\begin{align*}
\tau_{[a,b]} = \inf\Big\{x\in[a,b]: \eta(x)=\max_{u\in[a,b]}\eta(u) \Big\}.
\end{align*}
Then the distribution function of $\tau_{[a,b]}$ satisfies
\begin{align*}
F_{[a,b]}(x) = F_{[0,b-a]}(x-a), \quad \forall\, x\in\RR.
\end{align*}
Further, the law of $\tau_{[0,T]}$ restricted to $(0,T)$ is absolutely continuous with respect to the
Lebesgue measure on $\RR$ and its density has the following universal upper bound
\begin{align*}
f_T(t) \le \max\Big\{\frac{1}{t}, \frac{1}{T-t} \Big\},  \quad  0<t<T.
\end{align*}
\end{lemma}

\section{Gaussian tail estimates}\label{sec:tail_estimates}

In this section, we estimate the upper and lower tail probabilities for the associated Gaussian processes $\proo \eta$
in the second Ray-Knight Theorem. Recall that $\proo\eta$ is a centered Gaussian process with stationary increments satisfying
$$\rE[(\eta(x)-\eta(y))^2] = \sigma_0^2(x-y).$$

\sk
The first two lemmas give upper tail estimates for the maxima of $\eta$.
\begin{lemma}\label{lem: upper tail}
Assume that (C1) holds.
There exist finite positive constants $c_0$, $c_1$, $c_2$, $c_3$, $r_1$ such that for all $0<h<r_1$ and $u \ge c_3 \sigma_0(h)$,
\begin{align*}
\rE \bigg(\sup_{|x|\le h} |\eta(x)| \bigg)&\le c_0 \wh\sigma_0(h), \\
%\rE \bigg(\sup_{|x|\le h} \eta^2(x)\bigg) &\le  c_1 \wh\sigma_0^2(h), \\
\rP \bigg(\sup_{|x|\le h}|\eta(x)|>u\bigg) &\le c_1\exp \bigg(- \frac{u^2}{c_2\wh\sigma_0^2(h)}\bigg).
\end{align*}
\end{lemma}
\begin{proof}
%For Banach space valued Gaussian random variables, all the moments are equivalent (see, e.g. \cite[Cor 5.4.7]{MarcusRosen06}),
%the upper bound for $\rE\big(\sup_{|x|\le h} \eta^2(x)\big)$ follows from the first moment estimate.
By Marcus and Rosen \cite[Lem 7.2.2]{MarcusRosen06}, for any increasing $\wt\sigma(x)$ satisfying $\wt\sigma(0)=0$ and $\sigma_0(x)\le \wt\sigma(x)$,
\begin{align}\label{eq:lem31-1}
\rE \bigg(\sup_{|x|\le h} |\eta(x)| \bigg)\le C\bigg(\wt\sigma(h) + \int_0^{1/2}\frac{\wt\sigma(hu)}{u\sqrt{\log 1/u}} \rd u\bigg).
\end{align}
Actually, \eqref{eq:lem31-1} holds with $\wt\sigma(\cdot)$ replaced by $\wh\sigma_0(\cdot)=\max_{|x|\le|\cdot|}\sigma_0(x)$
(see \cite[p.301]{MarcusRosen06}). Applying Lemma \ref{lem:on_conditions2} yields that
\begin{align}\label{eq:lem31-2}
\int_0^{1/2}\frac{\wh\sigma_0(hu)}{u\sqrt{\log 1/u}} \rd u \le  c\wh\sigma_0(h) \int_0^{1/2} \frac{u^{(\al_1-\ep)/2}}{u\sqrt{\log 1/u}} \rd u \le c\wh\sigma_0(h)
\end{align}
for all $h>0$ sufficiently small. Combining \eqref{eq:lem31-1} and \eqref{eq:lem31-2} yields the first moment estimate.

To derive the upper tail probability, consider $u>3c_0\wh\sigma_0(h)$ (we take $c_0 \ge 1/\sqrt{2 \pi}$) and let $a$ be
the median of $\sup_{|x|\le h} \eta(x)$ which satisfies
$$\Big|a- \rE\sup_{|x|\le h} \eta(x) \Big|\le \wh\sigma_0(h)/\sqrt{2\pi},$$ see \cite[Cor. 5.4.5]{MarcusRosen06}. Using the
concentration inequality (cf. \cite[Rem. 5.4.4]{MarcusRosen06}) in the third inequality below, one gets
\begin{align*}
\rP \bigg(\sup_{|x|\le \de} |\eta(x)| > u \bigg)  &\le  \rP\bigg(\sup_{|x|\le \de} |\eta(x)|  -a > u - c_0\wh\sigma_0(h) - \wh\sigma_0(h)/\sqrt{2\pi}\bigg) \\
&\le \rP\bigg( \Big|\sup_{|x|\le \de} |\eta(x)|  -a \Big|\ge u/3 \bigg) \le 3 \bigg(1- \Phi(\frac{u}{3\wh\sigma_0(\de)})\bigg)\\
&\le 3\exp\bigg(-\frac{u^2}{18\wh\sigma_0^2(\de)} \bigg),
\end{align*}
where $\Phi$ is the distribution function of $N(0,1)$.  For smaller $u$, this inequality holds trivially with $3$
replaced by some generic constant.   $\square$
\end{proof}

\begin{lemma}\label{lem: upper tail large}
Assume that (C2) holds. There exists $K>0$ such that for all $h>K$, the conclusion in Lemma \ref{lem: upper tail}
holds with possibly different constants $c_0, c_1, c_2, c_3$.
\end{lemma}
\begin{proof}\,
Let $K_1$ be the constant in Lemma \ref{lem:on_conditions4} and $h>K_1$. As was shown in the proof of
Lemma \ref{lem: upper tail},  the upper tail probability estimate follows once we establish the first moment estimate
for the maxima of $\proo \eta$. Following Marcus \cite{Marcus01}, we use Dudley's entropy bound
(\cite[Thm. 6.1.2]{MarcusRosen06})
\begin{align*}
\rE \bigg(\sup_{|x|\le h}{\eta(x)} \bigg)\le 16\sqrt{2}\int_0^{\wh\sigma_0(h)} \sqrt{\log N([-h,h], d_\eta, u)}\, \rd u,
\end{align*}
where $d_\eta(x,y)= \sigma_0(x-y)$ is the canonical metric of the Gaussian process $\{\eta(x); $ $x \in \R\}$ and $N([-h,h],d_\eta,u)$ is the
minimal number of balls with radius at most $u$ in the metric $d_\eta$ to cover the interval $[-h,h]$. Observe that
{by the stationarity of the increments,}
%the metric  $d_\eta$ is homogeneous in the sense that
$d_\eta(x,y)= d_\eta(x-y,0)$,  thus $N([-h,h], d_\eta, u)\le (h/K_1) N([-K_1,K_1], d_\eta, u])$.
For any increasing function $\wt\sigma(x)$ that satisfies $\sigma_0(x)\le\wt\sigma(x)$, one has by the subadditivity of the square root function
\[
\begin{split}
&\int_0^{\wt\sigma(K_1)} \sqrt{\log N([-h,h], d_\eta, u)} \rd u \\
&\qquad \le  \wt\sigma(K_1)\sqrt{\log (h/K_1)} + \int_0^{\wt\sigma(K_1)} \sqrt{\log N([-1,1], d_\eta, u)} \rd u.
\end{split}
\]
That the last integral is finite follows from Fernique's Theorem \cite[Thm. 6.2.2]{MarcusRosen06}.
%because it is a sufficient condition for the continuity of $\proo\eta$, and the continuity of the associated Gaussian process is equivalent to the joint continuity of the local times, see \cite[Thm 9.4.1 and Thm 6.1.2]{MarcusRosen06}.
 Also, a $d_\eta$-ball of radius $u$ covers at least an interval of Euclidean length ${\wt\sigma}^{-1}(u)$ because
 $[0,\wt\sigma^{-1}(u)]=\bra{x\ge 0: \wt\sigma(x)\le u}\subset\bra{x\ge 0:\sigma_0(x)\le u}$. Hence,
\begin{multline}\label{eq: entropy large r}
\int_{\wt\sigma(K_1)}^{\wt\sigma(h)}  \sqrt{\log N([-h,h], d_\eta, u)} \rd u \le \int_{\wt\sigma(K_1)}^{\wt\sigma(h)}  \sqrt{\log \frac{2h}{{\wt\sigma}^{-1}(u)}} \rd u \\
= \int_{K_1}^h\sqrt{\log {2h \over s}} \rd\wt\sigma(s) \le \wt\sigma(h)(\log 2)^{1/2} + {1\over 2} \int_{K_1/(2h)}^{1/2} \frac{\wt\sigma(2hs)}{s\sqrt{\log(1/s)}}\rd s.
\end{multline}
Since $\sigma_0(x-y)=d_\eta(x,y)\le d_\eta(x, (x+y)/2) + d_\eta((x+y)/2, y) = 2\sigma_0((x-y)/2)\le 2\wh\sigma_0((x-y)/2)$,   one gets for all $x$,
$$
\wh\sigma_0(2x)\le 2\wh\sigma_0(x).
$$
Taking limits (through a sequence of finite measures $\rd \wt\sigma_n$ on $[1,h]$ that converges to $\rd\wh\sigma_0$),
one obtains \eqref{eq: entropy large r} with $\wt\sigma$ replaced by $\wh\sigma_0$.  To summarize, we have proved that
\begin{align*}
\rE\bigg(\sup_{|x|\le h}{\eta(x)}\bigg) &\le C\left(\sqrt{\log h} + \wh\sigma_0(h) +  \int_{K_1/(2h)}^{1/2} \frac{\wh\sigma_0(hs)}{s\sqrt{\log 1/s}} \rd s\right)\\
&:= C(A_1+A_2+A_3).
\end{align*}
It  follows from Lemma \ref{lem:on_conditions4} that $A_1<cA_2$. Similarly,
\begin{align*}
A_3 \le c\wh\sigma_0(h)\int_0^{1/2} \frac{s^{\be_1-\ep}}{s\sqrt{\log 1/s}}\rd s \le c\wh\sigma_0(h).
\end{align*}
This ends the proof. $\square$
\end{proof}

Next, we consider the lower tail probability for the maxima of $\proo \eta$. Lemma \ref{lem:lowertail}  is the key technical estimate of this paper.
As suggested by Marcus and Rosen \cite[p.527]{MarcusRosen06}, we explore Molchan's idea \cite{Molchan99,Molchan00}
without using scaling.
%Remark that no monotonicity is needed for $\sigma_0^2(x)$.
\begin{lemma}\label{lem:lowertail}
Assume that (C1) holds. For all $\ga< 1/(\ov\al-1)$, there exists a finite constant $K_2$ such that for any $h<1$ and $\de<1$,
\begin{align*}
\rP \Big(\eta(x)<\sqrt{h\wh\sigma_0^2(\de)}; \, \hbox{ for all } |x|\le \de\Big) \le K_2 h^{\ga}.
\end{align*}
\end{lemma}
\begin{proof}\,
Set
\begin{align*}
\xi_\de(x)= \frac{\eta(\de x)}{\wh\sigma_0(\de)}; \quad |x|\le 1.
\end{align*}
The probability in question becomes
\begin{align*}
\rP(\xi_\de(x)<\sqrt{h};\, \hbox{ for all } |x|\le 1).
\end{align*}
It is plain that for each $0<\de<1$,  the mean zero Gaussian process $\{\xi_\de(x), |x|\le 1\}$ has stationary increments.
Furthermore, its variogram has the following spectral representation:
\begin{align*}
\rE \big[(\xi_\de(y+x)-\xi_\de(y))^2\big]={\sigma_0^2(\de x) \over \wh\sigma^2_0(\de)} =
\frac{2}{\pi}\int_0^\infty \frac{1-\cos(x\la)}{\de\wh\sigma_0^2(\de)\psi(\la/\de)}\rd\la.
\end{align*}
Thus, by Section 11.7 in Marcus and Rosen \cite{MarcusRosen06} %(see also Section 11.3),
the reproducing kernel Hilbert space of $\{\xi_\de(x), |x|\le 1\}$ is endowed with the norm
\begin{align*}
\|f-f(0)\|^2 = \frac{1}{2\pi}\int |\de\wh\sigma_0^2(\de)\psi(\la/\de)| \, |\wh f(\la)|^2\rd \la.
\end{align*}

Let $f$ be a non negative smooth function with support in $[-1,1]$ and $f(0)=1$.  Set $f_h(x)= \sqrt{h}f(x/h^{a})$ and
 $\bar f_h(x):=\sqrt{h}-f_h(x)= f_h(0)-f_h(x)$, where $a>1$ is to be chosen later.  The Cameron-Martin
 formula (Theorem \ref{CM}) applied to $\{\xi_\de(x), |x|\le 1\}$ yields
\begin{align*}
&\rP \Big(\xi_{\de}(x)< \sqrt h; |x|\le 1\Big) \\
&= \rP \Big(\xi_{\de}(x) - \bar f_h(x) < f_h(x); |x|\le 1\Big) \\
&= e^{-\|\bar f_h\|^2/2} \rE \Big[e^{-\xi_{\de}(\bar f_h) }{\mathbf 1}_{\{\xi_{\de}(x)<f_h(x);|x|\le 1\}}\Big],
\end{align*}
where ${\mathbf 1}_{F}$ is the indicator of the event $F$.
By using H\"older's inequality, one bounds the last term from above by
\begin{align}\label{eq:lem3.3_1}
e^{(p-1)\|\bar f_h\|^2/2} \rP \big(\xi_{\de}(x)<f_h(x);|x|\le 1\big)^{1/q},
\end{align}
where $1/p+1/q=1$ with $p,q>1$.

Below we show that there exists a finite constant $M$ (depending only on $f$) such that
\begin{align}\label{eq: invariant norm}
\|\bar f_h \|^2 \le M
\end{align}
and
\begin{align}\label{eq: small max location}
\rP \big(\xi_{\de}(x)<f_h(x);|x|\le 1 \big) \le M h^a \mbox{ with } a=1/(\ov\al+\ep-1)
\end{align}
for all $\de, h<1$. As $q$ can be chosen arbitrarily close to 1, and $\ep$ arbitrarily small, the desired estimate follows.

\mk
To show \eqref{eq: invariant norm},  observe that  $\hat f_h(\la)= h^{{1\over 2}+a}\hat f(\la h^{a})$ and $\|\bar f_h\|
= \|f_h - f_h(0)\|$, thus
\begin{align}
\|\bar f_h\|^2 & = \frac{1}{\pi}\int_0^\infty |\de\wh\sigma_0^2(\de)\psi(\la/\de)| \,|\hat f_h(\la)|^2\rd \la \nonumber\\
& = \frac{1}{\pi}\de\wh\sigma_0^2(\de) h^{1+2a}\int_0^\infty \psi(\la/\de)|\hat f(\la h^a)|^2\rd \la \nonumber\\
&=  \frac{1}{\pi} \de\wh\sigma_0^2(\de) h^{1+a}\int_0^\infty \psi\left(\frac{\la}{\de h^a}\right)|\hat f(\la)|^2\rd \la. \label{eq:lem3.3_2}
\end{align}

We split the last integral into two parts.  For $\la\le K_0 \de h^a$ (recalling that $K_0$ is the constant in Lemma
\ref{lem:on_conditions}), the integrand concerning $\psi$ is bounded due to the continuity of $\psi$, then
\begin{align*}
\int_0^{K_0 \de h^a} \psi\left(\frac{\la}{\de h^a}\right)|\hat f(\la)|^2\rd\la \le c\int_0^\infty |\hat f(\la)|^2\rd\la =c \|f \|_{L^2},
\end{align*}
since Fourier transform is an isometry from $L^2(\RR)$ to $L^2(\RR)$. For $\la> K_0 \de h^a$, we use Lemma \ref{lem:on_conditions},
\begin{align*}
\int_{ K_0 \de h^a}^\infty \psi\left(\frac{\la}{\de h^a}\right)|\hat f(\la)|^2\rd\la \le {c \psi\left(\frac{1}{\de h^a}\right)}
 \int_0^\infty |\la|^{\ov\al+\ep}|\hat f(\la)|^2\rd\la,
\end{align*}
where we can assemble the last integral into the constant $c$ because $\wh f$ is a Scharwz function. Therefore,
one obtains for $h<1$ and $\de$ sufficiently small,
\begin{align*}
\|\bar f_h\|^2 \le c \de\wh\sigma_0^2(\de) h^{1+a}  \psi\left(\frac{1}{\de h^a}\right)\le  c\de\wh\sigma_0^2(\de)  \psi(1/\de) h^{1+a-a(\ov\al+\ep)},
\end{align*}
where we used again Lemma \ref{lem:on_conditions} in the second inequality.

By Lemma \ref{lem:on_conditions2}, one has
\begin{align*}
\de\wh\sigma_0^2(\de)\psi(1/\de) \le c \de\int_{1/\de}^{\infty} \frac{\psi(1/\de)}{\psi(\la)}\rd \la \le c\de \int_{1/\de}^\infty
\left( 1\over \de\la\right)^{\un\al-\ep} \rd\la =c.
\end{align*}
Letting $a=1/(\ov\al+\ep-1)$ yields \eqref{eq: invariant norm}.

\mk
Now we prove \eqref{eq: small max location}. Denote by $\tau_{\de}$ the leftmost maximum location of the process $\{\xi_{\de}(x), |x|\le 1\}$.
Since $f_h$ vanishes outside of $[-h^{1/(\be-1)},h^{1/(\be-1)}]$, one has
\begin{align*}
\rP\big(\xi_{\de}(x)<f_h(x);|x|\le 1\big) \le \rP \big(|\tau_{\de}|\le h^{1/(\be-1)}\big).
\end{align*}
By Lemma \ref{lem:loc_max}, the density at zero of the law of each $\tau_{\de}$ is bounded from above by $2$,
uniformly for all $\de>0$. This implies \eqref{eq: small max location} and completes the proof of the  lemma.
$\square$
\end{proof}

We end this section with a similar lower tail estimate, but for large $\de$.

\begin{lemma}\label{lem:lowertail_large}
Assume that (C2) holds. For all $\ga< 1/(\ov\be-1)$, there exists a finite constant $K_3>1$ such that for any $\de>1$
and $h<1$ so that $\de h^{1/(\ov\be+\ep-1)}>1/r_1$ (the constant defined in Lemma \ref{lem:on_conditions_3}),
\begin{align*}
\rP \Big(\eta(x)<\sqrt{h\wh\sigma_0^2(\de)}; |x|\le \de\Big) \le K_3 h^{\ga}.
\end{align*}
\end{lemma}

\begin{remark}
{\rm This lemma merely says that when $h\to 0$ and  $\de\to\infty$ in such a way that $\de h^a \to \infty$ for some $a<1/(\ov\be-1)$,
the lower tail probability behaves like a power function of $h$. }
\end{remark}
\noindent{\it Proof of Lemma \ref{lem:lowertail_large}}\,
Since \eqref{eq:lem3.3_1} holds for all $h,\de>0$, we only need to show that there exists a finite constant $M$ such that
\eqref{eq: invariant norm} and \eqref{eq: small max location} hold with $a=1/(\ov\be+\ep-1)$, uniformly for all $h<1$ and
all $\de>1$ so that $\de h^{1/(\ov\be+\ep-1)}>1/r_1$.

Let us start with \eqref{eq: invariant norm}. One has as in \eqref{eq:lem3.3_2} that
\begin{align*}
\|\bar f_h\|^2 =  \frac{1}{\pi} \de\wh\sigma_0^2(\de) h^{1+a}\int_0^\infty \psi\left(\frac{\la}{\de h^a}\right)|\hat f(\la)|^2\rd \la.
\end{align*}
Write
\begin{multline*}
\int_0^\infty \psi\left(\frac{\la}{\de h^a}\right) |\hat f(\la)|^2\rd \la \\
= \int_0^1 +\int_1^{r_1\de h^a}
+ \int_{r_1\de h^a}^\infty \psi\left(\frac{\la}{\de h^a}\right)|\hat f(\la)|^2\rd \la
 =: B_1+B_2+B_3.
\end{multline*}
By Lemma \ref{lem:on_conditions_3}, one gets
\begin{align*}
B_1 &\le \psi\left(\frac{1}{\de h^a}\right)\int_0^1 \la^{\un\be-\ep} |\wh f(\la)|^2\rd \la \le c\psi\left(\frac{1}{\de h^a}\right), \\
B_2 &\le \psi\left(\frac{1}{\de h^a}\right)\int_1^{r_1\de h^a} \la^{\ov\be+\ep} |\wh f(\la)|^2\rd \la \le  c\psi\left(\frac{1}{\de h^a}\right).
\end{align*}
Further, a variable change,  $\psi(\la)=o(\la^2)$ at infinity (\cite[Lem. 4.2.2]{MarcusRosen06}), and the fact that $\wh f$ is
a Schwarz function implies that
\begin{align*}
B_3 &= \de h^a \int_{r_1}^\infty \psi(\la) |\wh f(\de h^a\la)|^2 \rd\la \\
&\le c \de h^a \int_{r_1}^\infty  \la^2 |\de h^a\la|^{-k}\rd \la = c(\de h^a)^{1-k}
\end{align*}
for arbitrarily large $k\in\NN$.  Setting $k=3$ and using Lemma \ref{lem:on_conditions_3}, one obtains
$(\de h^a)^{-2} \le c(\de h^a)^{-\ov\be-\ep} \le \psi(1/(\de h^a))$, which implies $B_3\le c \psi\left(1/({\de h^a})\right)$. Combining these estimates yields that
\begin{align*}
\|\bar f_h\|^2 \le c  \de\wh\sigma_0^2(\de) h^{1+a}\psi\left(\frac{1}{\de h^a}\right) \le c\de \wh\sigma_0^2(\de)\psi(1/\de) h^{1+a-a(\ov\be+\ep)},
\end{align*}
where we have used Lemma  \ref{lem:on_conditions_3} again in the second inequality. Applying Lemma \ref{lem:on_conditions_3} as in the proof
of Lemma \ref{lem:on_conditions4}  entails that $\de \wh\sigma_0^2(\de)\psi(1/\de)\le c$ for all $\de$ large enough. Letting $a=1/(\ov\be+\ep-1)$
gives \eqref{eq: invariant norm}.

Remark that, by Lemma \ref{lem:loc_max}, the uniform density bound for the maximum location $\tau_\de$ holds also for large $\de$.
Repeating the same lines as in the proof of Lemma \ref{lem:lowertail} yields \eqref{eq: small max location} which completes the proof.
$\square$
%\end{proof}

\section{Proofs of Theorems \ref{theo} and  \ref{Th:Polar} }
\label{sec:proof_theo}

Following \cite{BassGriffin85,BassEisenbaumShi00,Marcus01}, the strategy is to first find an upper function for the local
times with relatively small space variable, then to obtain a lower function for the  maximal local times, where the order
of the upper and lower functions are the same. This would give enough information to derive asymptotic results for
the favorite points. Precisely, we prove the following two lemmas.

\begin{lemma}\label{lem:small_x} Let $h_a(t)= \phi^{-1}\left( {t\over (\log 1/t)^a} \right)$ with $a>0$.
\begin{itemize}
\item[(i)] Assume (C1).
  For all $\ga<a/2$,
$$\lim_{t\to 0} \sup_{|x|\le h_a(t)} \frac{(\log 1/t)^\ga (L^x_{\tau(t)}-t)}{t}=0 \quad \mbox{a.s.}$$
\item[(ii)] Assume (C2). For all $\ga<a/2$,
$$\lim_{t\to \infty} \sup_{|x|\le h_a(t)} \frac{(\log 1/t)^\ga (L^x_{\tau(t)}-t)}{t}=0 \quad \mbox{a.s.}$$
\end{itemize}
\end{lemma}

\begin{lemma}\label{lem:all_x}
 Set $L^*_t= \sup_{x} L^x_t$.
\begin{itemize}
\item[(i)]  Assume (C1). For any $\ga>(\ov\al-1)/(2-\ov\al)$,
\begin{align*}
\lim_{t\to 0} \frac{(\log 1/t)^\ga(L^*_{\tau(t)}-t)}{t}=\infty \quad \mbox{a.s.}
\end{align*}
\item[(ii)]  Assume (C2). For any $\ga>(\ov\be-1)/(2-\ov\be)$,
\begin{align*}
\lim_{t\to \infty} \frac{(\log 1/t)^\ga(L^*_{\tau(t)}-t)}{t}=\infty \quad \mbox{a.s.}
\end{align*}
\end{itemize}
\end{lemma}

\noindent{\it Proof of Lemma \ref{lem:small_x}}\,
Let $0<\de<1$.  By Theorem \ref{RK},
\begin{align*}
&\PP \bigg(\sup_{0\le x\le h_a(t)} L^x_{\tau(t)}  - t \ge (1+\ep)\la\bigg) \\
&\le  \rP\times\PP \bigg(\sup_{0\le x\le h_a(t)} \Big( L^x_{\tau(t)}  + {\eta^2(x) \over 2} - t  \Big)\ge (1+\ep)\la \bigg) \\
&= \rP\bigg(\sup_{0\le x\le h(t)}  \Big({\eta^2(x) \over 2}  + \sqrt{2t} \eta(x) \Big) \ge (1+\ep)\la\bigg) \\
&\le \rP \bigg(\sup_{0\le x\le h_a(t)} {\eta^2(x)} \ge 2\ep\la\bigg)  + \rP\bigg(\sup_{0\le x\le h_a(t)}  \sqrt{2t} |\eta(x)| \ge \la\bigg)
 := P_1+P_2.
\end{align*}
Set
$$
\la = \sqrt{{2t \cdot c_3\wh\sigma_0^2(h_a(t))(1+\ep)\log\log 1/t} \over {\de}}.
$$
 Then the upper tail estimate in Lemma \ref{lem: upper tail} yields,
\begin{align*}
 P_2 \le  c_2\exp\left( - \frac{(1+\ep)\log\log 1/t}{\de} \right).
\end{align*}
By Lemma \ref{lem:on_conditions2},  $\wh\sigma^2_0(h_a(t))\asymp (\log 1/t)^{-a}$ around zero,  one deduces
\begin{align*}
P_1 &=\rP \bigg(\sup_{0\le x\le h_a(t)} |\eta(x)|\ge \sqrt{2\ep\la}\bigg) \\
&\le c_2\exp \left( - 2\ep \sqrt{{2t(1+\ep)\log\log 1/t \over c_3\wh\sigma_0^2(h_a(t))\de}}  \right) \\
&= c_2\exp \left(  - c\ep \sqrt{ (1+\ep) \log\log 1/t \over \de/(\log 1/t)^a }  \right) \\
&\le  c_2\exp\left( - \frac{(1+\ep)\log\log 1/t}{\de} \right)
\end{align*}
for $t$ sufficiently small.

Let $t_k= \exp(-k^\de)$.  When $t=t_k$,  both $P_1$ and $P_2$ are the general term  of a convergent series.
By the Borel-Cantelli Lemma,  a.s. for all $k$ large enough,
\begin{align}\label{eq:p.17}
 \sup_{0\le |x|\le h_a(t_k)} {|\log t_k|^{a/2} (L^x_{\tau(t_k)}-t_k)  \over t_k \sqrt{\log\log 1/t_k}  } \le  c/\sqrt{\de}.
\end{align}
Thus, one has a.s. for all large $k$, and $t_{k-1}< t < t_k$
\begin{align*}
 \sup_{0\le |x|\le h_a(t_k)} {|\log t_k|^{a/2} (L^x_{\tau(t)}-t)  \over t_k \sqrt{\log\log 1/t_k}  } \le  c/\sqrt{\de}
 + { (t_k-t_{k-1})|\log t_k|^{a/2} \over  t_k \sqrt{\log|\log t_k|}},
\end{align*}
where the second term is bounded from above by $\sqrt{\de}$.  Therefore,  a.s.
\begin{align*}
\sup_{|x|\le h_a(t)} {|\log t|^{a/2} (L^x_{\tau(t)}-t)  \over t \sqrt{\log\log 1/t}  } \le \left(  {c\over \sqrt{\de}}
+\sqrt\de  \right) { t_k\log\log 1/t_k \over   t_{k-1} \log\log t_{k-1} } \le 2\left(  {c\over \sqrt{\de}} +\sqrt\de  \right)
\end{align*}
for all  $t$ sufficiently small.  The first claim follows.

To show the second claim, the proof is identical except that we use Lemma \ref{lem: upper tail large} instead of
Lemma \ref{lem: upper tail}, and that we choose $t_k = \exp(k^\de)$.  We omit the details.
$\square$
%\end{proof}
\medskip

\noindent{\it Proof of Lemma \ref{lem:all_x}}\,
Define the events
\begin{align*}
E_1 &= \bra{L^x_{\tau(t)}-t\le C_1th; |x|\le \phi^{-1}(ht)};\\
E_2 &= \bra{{\eta^2(x)\over 2}\le C_1th; |x|\le \phi^{-1}(ht)};\\
E_3 &=\bra{L^x_{\tau(t)}+ {\eta^2(x)\over 2}- t\le 2C_1th; |x|\le \phi^{-1}(ht)}.
\end{align*}
By the Markov inequality and the second moment estimate in Lemma \ref{lem: upper tail}, we can choose the constant
$C_1$ in the event $E_2$ large enough such that $\rP(E_2)>1/2$ uniformly in $h,t<1$.  Observe that by the
independence and Theorem \ref{RK}, one has
\begin{align*}
{\PP(E_1)\over 2}\le \PP(E_1) \rP(E_2)&\le \rP\times\PP(E_3)\\
&= \rP\left({\eta^2(x)\over 2}+\sqrt{2t}\eta(x)\le 2C_1th;|x|\le \phi^{-1}(ht) \right)\\
&\le \rP\left(\eta(x)\le \sqrt{2}C_1\sqrt{th}\sqrt{h};|x|\le \phi^{-1}(ht) \right).
\end{align*}
Recall that by Lemma \ref{lem:on_conditions2}, $\wh\sigma_0^2(\phi^{-1}(\de))\asymp \de$ for $\de$ sufficiently small,
in other words, we can choose $C_2<\infty$ such that
$$\sqrt{2}C_1\de\le C_2\wh\sigma_0(\phi^{-1}(\de^2)).$$
In particular, for all $\ga<1/(\ov\al-1)$ and $t,h$ sufficiently small, one has by Lemma \ref{lem:lowertail} that
\begin{align*}
\PP(E_1)\le 2 \rP\left(\eta(x)\le C_2\sqrt{h\wh\sigma_0^2(\phi^{-1}(th))} ;|x|\le \phi^{-1}(ht) \right)\le K h^{\ga}
\end{align*}
with some finite constant $K$. It follows that
\begin{align*}
\PP(L^*_{\tau(t)}-t\le C_1th) \le Kh^{\ga}.
\end{align*}

Now let $0<d<2-\ov\al$, $t=t_k:= \exp(-k^d)$ and $h=h_k:= k^{-(\ov\al-1)(1+\ep)}$ with $0<\ep<(1-d)/(\ov\al-1)-1$.
The Borel-Cantelli Lemma implies that almost surely for all $k$ sufficiently large,
\begin{align*}
L^*_{\tau(t_k)}-t_k  \ge \frac{C_1t_k}{k^{(\ov\al-1)(1+\ep)}}.
\end{align*}
Let $t_{k}\le t<t_{k-1}$, then,
\begin{align*}
L^*_{\tau(t)}-t \ge -(t-t_k) + \frac{C_1t_k}{k^{(\ov\al-1)(1+\ep)}}
\end{align*}
and
\begin{align*}
{L^*_{\tau(t)}-t \over t} \ge -{(t_{k-1}-t_{k})\over t_{k}} + \frac{C_1t_k}{t_{k-1}k^{(\ov\al-1)(1+\ep)}} \ge -2k^{-(1-d)}
+ {C_1(1-\de) \over k^{(\ov\al-1)(1+\ep)}}
\end{align*}
for all $\de>0$ and large $k$. By our choice of $d$ and $\ep$, one deduces
\begin{align*}
{L^*_{\tau(t)}-t \over t} \ge {C_1(1-2\de) \over k^{(\ov\al-1)(1+\ep)}} \ge { C_1(1-3\de)\over (\log 1/t)^{(\ov\al-1)(1+\ep)/d} }.
\end{align*}
As $d$ can be chosen arbitrarily close to $2-\ov\al$, the first claim of Lemma \ref{lem:all_x} follows.

Similarly, applying the second Ray-Knight Theorem and Lemma \ref{lem:lowertail_large} yields
$$\PP(E_1) \le Kh^{\ga}$$
for $h<1$, $\de>1$ such that $\de h^{1/(\ov\be+\ep-1)}>1/r_1$. Repeating the arguments in the last paragraph for
$t_k=\exp(k^d)$ and $h_k= k^{-(\ov\be-1)(1+\ep)}$ (with some tuning in the choice of $\ep$) gives the second claim of Lemma \ref{lem:all_x}.
$\square$
%\end{proof}

\sk
We are ready to finish the proof of Theorem  \ref{theo}.

\sk
\noindent{\it Proof of Theorem \ref{theo}.}\,
 We only proof part (i) asymptotic behavior of $\pro V$ around zero, the proof of (ii) is similar. Combining Lemma \ref{lem:small_x} (i) and
 Lemma \ref{lem:all_x} (i) entails that for $\ga>2(\ov\al-1)/(2-\ov\al)$,
\begin{align}\label{eq:V_tau_small}
V_{\tau(t)}\ge \phi^{-1}\left( \frac{t}{(\log t)^\ga}\right)
\end{align}
for all $t$ sufficiently small. In other words, for all small $t$,
\begin{align*}
V_t \ge \phi^{-1}\left( \frac{L^0_t}{(\log L^0_t)^\ga}\right),
\end{align*}
which implies that
\begin{align*}
 \lim_{t\to 0}  \frac{ V_t}{\phi^{-1}\left( \frac{L^0_t}{(\log L^0_t)^\ga}\right)}=\infty.
\end{align*}
%To obtain a deterministic lower function for $\pro V$ around zero, one needs to find an  upper function for the
%subordinator $\tau(t) = \inf\{ s\ge 0: L^0_s >t\}$, then use \eqref{eq:V_tau_small}.  By Prop. 4, Chap. V in
% \cite{Bertoin96book},  the Laplace exponent of $\bra{\tau(t);t\ge 0}$ is for $\la>0$,
%\begin{align*}
%\EE[\exp(-\la\tau(t))] = \exp\bra{-t/u^\la(0)}, \quad t>0.
%\end{align*}
%Recall that (page 140 of \cite{MarcusRosen06}) the $\la$-potential density $u^\la$ of  a
%L\'evy process with characteristic exponent $\psi$ satisfies
%$$u^\la(0) = \frac{1}{\pi}\int_0^\infty \frac{1}{\la+\psi(x)}\rd x.$$
This finishes the proof of Theorem \ref{theo}.
$\square$
%\end{proof}

\sk
We end this section by proving Theorem  \ref{Th:Polar}.

\sk
\noindent{\it Proof of Theorem  \ref{Th:Polar}.}\,
By Theorem 1.3 and Proposition 1.4 of Eisenbaum and Khoshnevisan \cite{EKh02},  it is sufficient to verify that
there is a sequence $\{x_n\}$ such that $\lim\limits_{n \to \infty}x_n = 0$ and
\begin{equation}\label{Con:EKo2}
\lim_{k \to \infty} (\ln k)^{1/2} \sup_{\tiny{\begin{array} {ll}&n,m \in \mathbb N,\\
 &|n-m|\ge k \end{array}}} \frac{u_{T_0}(x_m, x_n)}{\sqrt{u_{T_0}(x_m, x_m) u_{T_0}(x_n, x_n)}} = 0,
\end{equation}
where $u_{T_0}(x,y)$ is the function given in \eqref{Def:u}. The verification is similar to the proof of Theorem 5.2
in \cite{EKh02} and we give it for the sake of completeness.  It follow from  \eqref{Def:u}  that for any $x, y \in \R$,
\begin{equation}\label{Eq:EK3}
\frac{u_{T_0}(x, y)}{\sqrt{u_{T_0}(x, x) u_{T_0}(y, y)} }\le \frac 1 2 \frac{\sigma_0(x)}{\sigma_0(y)}
+ \frac{|\sigma_0^2(y)- \sigma_0^2(x-y)|}{\sigma_0(x)\sigma_0(y)}.
\end{equation}
Under Condition (C1), {we apply Remark \ref{rem:2.4}} to see that for
any $0 < x < y \le 1$,%\footnote{This needs to be checked. Maybe we should add a lemma in Section 2,
%which is similar to Lemma 2.3 in the above, but it is for $\sigma_0^2(x).$ X: done.}
\begin{equation}\label{Eq:EK4}
\frac{\sigma_0(x)}{\sigma_0(y)} \le \bigg(\frac x y\bigg)^{\alpha - \varepsilon},
\end{equation}
where $\alpha> 0$ is a constant and $\varepsilon \in (0, \alpha)$. In order to bound the second term on the right hand side of
\eqref{Eq:EK3}, we use the fact that the covariance function $u_{T_0}(x,y)$ of the associated Gaussian process
 is the $0$-potential density of $X$ killed the first time it hits zero.  By \cite[Formula (7.232), p.324]{MarcusRosen06},
\begin{align*}
\sigma_0^2(x)\ge |\sigma_0^2(y) - \sigma_0^2(y-x)|.
\end{align*}
Therefore, the second term is bounded from above by $\frac{\sigma_0(x)}{\sigma_0(y)}$.  We have thus proved
\begin{align*}
\frac{u_{T_0}(x, y)}{\sqrt{u_{T_0}(x, x) u_{T_0}(y, y)} }\le c  \bigg(\frac x y\bigg)^{\alpha - \varepsilon}.
\end{align*}

Once this is in place, the rest of the proof is almost the same as that in 
the bottom on page 254 of \cite{EKh02} and is omitted. $\square$
% \end{proof}

 \section{Examples}\label{sec:example}

{In this section, we provide more examples of symmetric L\'evy processes 
(besides the semistable L\'evy processes in \cite{choi1994} that we have 
mentioned earlier) that satisfy the conditions of the present paper, 
but not those in Marcus \cite{Marcus01}.} To avoid repetition, we only 
present examples concerning the asymptotic behavior of the favorite 
points around zero. 

%As we mentioned earlier, the characteristic exponent of a general 
%semistable L\'evy process is not regularly varying at zero (cf. Choi ),
%hence it does not satisfy the conditions in Marcus \cite{Marcus01}.
%\blue{In the following we provide more examples.}

\begin{example}\,
Let $1<\al<2$ and $0<c_1< c_2$ be constants such that $c_2\al/(2c_1)<2$ 
and $c_1\al/(2c_2)>1$. For any decreasing sequence $\{b_n; n\ge 0\}$ such 
that $b_0=1$ and $b_n\to 0$,  define
\begin{align*}
\theta(x)= \begin{cases} c_1|x|^{\al+1}, & |x|\in (b_{2k+2},b_{2k+1}],\\
 c_2|x|^{\al+1}, & |x|\in (b_{2k+1},b_{2k}].
\end{cases}
\end{align*}
Then (C1) holds with $\un\al= c_1\al /(2c_2)< c_2\al/(2c_1)=\ov\al$.
\end{example}

\begin{example}\,
Let $1<\al_1<\al_2<2$ and  $\{b_n; n\ge 0\}$ be a real sequence decreasing to zero such that $b_0=1$, $b_1={1 \over 2}$,
and for any integer $k\ge 1$,
\begin{align}
b_{2k}\le b_{2k-1}/(k+1)\label{eq:example52_0},\\
\int_{b_{2k+1}}^{b_{2k}} \frac{\rd x}{x^{\al_2+1}} \le \int_{b_{2k}}^{b_{2k-1}} \frac{\rd x}{x^{\al_1+1}}, \label{eq:example52_1}\\
\int_{b_{2k+1}}^{b_{2k}} \frac{\rd x}{x^{\al_2-1}} \le \int_0^{b_{2k+1}} \frac{\rd x}{x^{\al_1-1}}. \label{eq:example52_2}
\end{align}
This can be done inductively. We choose first $b_{2k}$ satisfying \eqref{eq:example52_0} which ensures the convergence as $k\to \infty$,
then chose $b_{2k+1}$ close to $b_{2k}$ so that \eqref{eq:example52_1} and \eqref{eq:example52_2} hold.  Let
\begin{align*}
\theta(x)= \begin{cases} |x|^{\al_1+1}, & |x|\in (b_{2k+2},b_{2k+1}],\\
 |x|^{\al_2+1}, & |x|\in (b_{2k+1},b_{2k}].
\end{cases}
\end{align*}
We claim that the following properties hold:
\begin{itemize}
\item[(i)] $\pi(\la)\asymp  |\la|^{\al_1}$ for $|\la|$ large enough;
\item[(ii)] $ \psi(\la)\asymp |\la|^{\al_1}$ for $|\la|$ large enough;
\item[(iii)] Condition (C1) fails.
\end{itemize}
In particular, the conclusion of Lemma \ref{lem:on_conditions} holds 
with $\ov\al = \un\al=\al_1$ even if (C1) fails.
Consequently,  Part (i) of Theorem \ref{theo} still holds with $\ov\al=\al_1$.

Let us show the claims (i)-(iii). Observe by \eqref{eq:LK_formula} and 
the integrability of $1/\theta(x)$ at infinity ($\rd x/\theta(x)$
is a L\'evy measure) that
\begin{align}\label{eq:in_exmaple52}
c|\la|^{\al_1}\le \psi(\la) \le C|\la|^{\al_2}
\end{align}
for all $|\la|$ large enough. Similarly, 
$$c|\la|^{\al_1}\le \pi(\la)\le C|\la|^{\al_2}$$ 
for $|\la|$ sufficiently large. To show (i),
it suffices to show $\pi(\la)\le C|\la|^{\al_1}$ for large $|\la|$.  For any $\la>1$, there exists a constant $k_0$ such that either
 $1/\la\in (b_{2k_0+2}, b_{2k_0+1}]$ or $1/\la\in (b_{2k_0+1}, b_{2k_0}]$. In the first case,
\begin{multline*}
\pi(\la)/2 \le \frac{1}{\al_1}\left(\la^{\al_1} - b_{2k_0+1}^{-\al_1} + b_{2k_0}^{-\al_1} - b_{2k_0-1}^{-\al_1} + \cdots + b_2^{-\al_1} - b_1^{-\al_1}\right) \\
 + \frac{1}{\al_2}\left( b_{2k_0+1}^{-\al_2} - b_{2k_0}^{-\al_2} + b_{2k_0-1}^{-\al_2} - b_{2k_0-2}^{-\al_2} + \cdots+ b_1^{-\al_2}
 - b_0^{-\al_2} \right) +  \int_1^\infty \frac{\rd x}{\theta(x)}.
\end{multline*}
Plainly, the first sum is bounded from above by $C\la^{\al_1}$ and the integral is finite. It follows from \eqref{eq:example52_1}
that the second sum is bounded from above by $C\la^{\al_1}$.  In the second case when $1/\la\in (b_{2k_0+1}, b_{2k_0}]$,
\begin{multline}
\pi(\la)/2 \le \frac{1}{\al_1}\left( b_{2k_0}^{-\al_1} - b_{2k_0-1}^{-\al_1} + \cdots + b_2^{-\al_1} - b_1^{-\al_1}\right) \\
+ \frac{1}{\al_2}\left( b_{2k_0+1}^{-\al_2} - b_{2k_0}^{-\al_2} + b_{2k_0-1}^{-\al_2} - b_{2k_0-2}^{-\al_2} + \cdots+ b_1^{-\al_2} - b_0^{-\al_2} \right) +  \int_1^\infty \frac{\rd x}{\theta(x)}.\label{eq:example52_4}
\end{multline}
Again the both sums are bounded from above by $C\la^{\al_1}$ by \eqref{eq:example52_1} and the integral is finite. Hence, $\pi(\la)\le C \la^{\al_1}$ for $\la>1$, as desired.

Now we verify (ii). Recalling \eqref{eq:in_exmaple52} and Lemma \ref{lem:KS} (applied to the L\'evy measure rather than the spectral measure), we only need to  show that
\begin{align}\label{eq:example52_3}
\la^2\int_{0}^{1/\la} x^2\frac{\rd x}{\theta(x)}  \le C\la^{\al_1}
\end{align}
 for $\la$ large enough. Let $\la>1$ and $k_1$ be the integer so that either $1/\la\in (b_{2k_1+2}, b_{2k_1+1}]$ or $1/\la\in (b_{2k_1+1}, b_{2k_1}]$.
 In the first case, by \eqref{eq:example52_2}
\begin{multline*}
\la^2\int_{0}^{1/\la} x^2\frac{\rd x}{\theta(x)}\le \\ \la^2\left( \int_0^{1/\la} \frac{\rd x}{x^{\al_1-1}} +  \int_{b_{2k_1+3}}^{b_{2k_1+2}} \frac{\rd x}{x^{\al_2-1}}
+ \int_{b_{2k_1+5}}^{b_{2k_1+4}} \frac{\rd x}{x^{\al_2-1}}+\cdots  \right)\\
 \le \frac{1}{2-\al_1} \la^{\al_1}  + \frac{1}{2-\al_2} \la^{\al_1}\le C\la^{\al_1},
\end{multline*}
as desired. We prove similarly  in the second case that \eqref{eq:example52_3} holds.

Finally, it follows from \eqref{eq:example52_4} that
\begin{align*}
\limsup_{x\to 0}\frac{x/\theta(x)}{\nu(y: |y|\ge x)}=\infty.
\end{align*}
Thus (C1) is not satisfied.
 \end{example}

 \section{Extensions}\label{sec_ext}

\subsection{Transient case}\label{s:transient}

In this subsection we assume that the characteristic exponent $\psi$ of the pure jump symmetric L\'evy process $X$ satisfies
\begin{equation}\label{Con:tran}
\int_1^\infty \frac{d\lambda}{\psi(\lambda)} < \infty \ \ \hbox{ and }\ \ \ \int_0^1 \frac{d\lambda} {\psi(\lambda)} <\infty.
\end{equation}
Consequently, $X$  is transient,  its local times exist, and its $0$-potential density $u(x,y)$ is bounded and continuous.
Further, the support of potential operators of $X$ is $\RR$, it follows from \cite[p.55]{Bertoin96book} and the absolute continuity of potential operators that
\begin{equation}\label{Con:h_x}
h_x = \PP^x(T_0<\infty)>0, \quad \forall x\in\RR.
\end{equation}

We state the generalized second Ray-Knight Theorem for transient symmetric Markov processes due to Eisenbaum {\it et al.}
\cite{eisenbaum2000ray}, see also \cite[Thm. 8.2.3]{MarcusRosen06}.  Set $\tau^-(t)=\inf\bra{s: L^0_s\ge t}$.

\begin{theorem}\label{RK:tran}
Assume that $h_x=\PP^x(T_0<\infty)>0$ for all $x\in\RR$.
%\footnote{\cite{eisenbaum2000ray} imposed no condition on $h_x$ but \cite{MarcusRosen06} did. After checking the proof, it turns out such a condition is necessary, on p.1791 of \cite{eisenbaum2000ray} there is $u_{T_0}(x,y)\over h_xh_y$. }.
Let $\eta=\bra{\eta_x, x\in\RR}$ be a mean zero Gaussian process with covariance $u_{T_0}(x,y)$. Then
for any $t>0$ and any countable subset $D\subset\RR$, under $\PP^0\times\rP\times\rP_\rho$, in law
\begin{equation}
\bra{L^x_{\tau^-(t\wedge L^0_\infty)}+{\eta_x^2 \over 2}; x\in D}=\bra{{1\over 2}{(\eta_x+ h_x\sqrt{2(t\wedge\rho)})^2}; x\in D},
\end{equation}
where $\rho$ is an exponential random variable with mean $u(0,0)$.
\end{theorem}

We also need the following fact in \cite[p.26]{Bertoin96book} which says that the excessive functions of
a Markov process with absolutely continuous potential operators are lower semi-continuous.

\begin{lemma}\label{lsc}
$h_x$ is lower semi-continuous.
\end{lemma}

\sk
Now let us focus on deriving an asymptotic result for $V$ around zero. It suffices to prove analogue
of Lemma \ref{lem:small_x} (i) and \ref{lem:all_x} (i) for transient processes.  We reduce the proof
to the point where we can apply Gaussian tail estimates as in Section \ref{sec:proof_theo}.

\mk
\begin{lemma} Let $h_a(t)$ be as in Lemma \ref{lem:small_x}.  Under Condition (C1), for all $\ga<a/2$,
\begin{align*}
\lim_{t\to 0} \sup_{|x|\le h_a(t)} \frac{|\log t|^\ga (L^x_{\tau^-(t)}-h_x^2 t)}{t} = 0, \mbox{ a.s.}
\end{align*}
\end{lemma}
\begin{proof}\, Define the events
\begin{align*}
E_4 &= \bra{L^0_\infty>1, \rho>1}; \\
E_5 &=\bra{\sup_{|x|\le h_a(t\wedge L^0_\infty)} \frac{\eta^2(x)}{2} + h_x\sqrt{2(t\wedge\rho)}\eta(x)+ h_x^2(t\wedge\rho - t\wedge L^0_\infty)\ge (1+\ep)\la}; \\
E_6 &= \bra{\sup_{|x|\le h_a(t)} \frac{\eta^2(x)}{2} + \sqrt{2t}\eta(x) \ge (1+\ep)\la}.
\end{align*}
Note that $E_4\cap E_5\subset E_6$ for $0<t<1$.  By Theorem \ref{RK:tran},
\begin{multline*}
\PP\left(\sup_{|x|\le h_a(t\wedge L^0_\infty)} L^x_{\tau^-(t\wedge L^0_\infty)} - h_x^2 (t\wedge L^0_\infty)\ge (1+\ep)\la\right)
\le  \rP\times\rP_\rho\times\PP (E_5).
\end{multline*}
By the independence and the fact that $L^0_\infty$, $\rho$ are exponential with mean $0<u(0,0)<\infty$, we have for $0<t<1$,
\begin{align*}
\PP\left(\sup_{|x|\le h_a(t\wedge L^0_\infty)} L^x_{\tau^-(t\wedge L^0_\infty)} - h_x^2 (t\wedge L^0_\infty)\ge (1+\ep)\la\right) \le e^{2\over u(0,0)}\rP(E_6),
\end{align*}
where the Gaussian upper tail estimates have been applied. Set $t_k=\exp(-k^\de)$ with $0<\de<1$.
It follows that under Condition (C1),   a.s. for all $k$ sufficiently large, \eqref{eq:p.17} holds with
$L^x_{\tau(t_k)}-t_k$ replaced by $L^x_{\tau^-(t_k\wedge L^0_\infty)} - h_x^2(t_k\wedge L^0_\infty)$
and $h_a(t_k)$ replaced by $h_a(t_k\wedge L^0_\infty)$.  Since a.s. $L^0_\infty>0$, one deduces
\begin{align*}
\sup_{|x|\le h_a(t_k)} \frac{|\log t_k|^{a/2} (L^x_{\tau^-(t_k)} - h_x^2 t_k) }{t_k\sqrt{\log\log 1/t_k}} \le c/\sqrt{\de}
\end{align*}
for all $k$ sufficiently large. A routine interpolation ends the proof. $\square$
\end{proof}

\begin{lemma}
Recall $L^*_t = \sup_{x}L^x_t$. Under (C1), for $\ga>(\ov\al-1)/(2-\ov\al)$,
\begin{align*}
\lim_{t\to 0}\frac{|\log t|^{\ga}(L^*_{\tau^-(t)}-h_x^2 t)}{t} =\infty.
\end{align*}
\end{lemma}
\begin{proof}
Define the events
\begin{align*}
E_7 &= \bra{ L^x_{\tau^-(t\wedge L^0_\infty)} - h^2_x (t\wedge L^0_\infty) \le C_1 th; |x|\le \phi^{-1}(ht)}; \\
E_8 &=\bra{ L^x_{\tau^-(t\wedge L^0_\infty)} + \frac{\eta_x^2}{2} - h_x^2 (t\wedge L^0_\infty)< 2C_1th; |x|\le \phi^{-1}(ht) }; \\
E_9 &= \bra{ \frac{\eta(x)^2}{2} + h_x\sqrt{2(t\wedge \rho)} \eta_x + h_x^2(t\wedge \rho - t\wedge L^0_\infty)< 2C_1th; |x|\le \phi^{-1}(ht) };\\
E_{10} &= \bra{\frac{\eta(x)^2}{2} + h_x\sqrt{2t} \eta(x) \le  2C_1th; |x|\le \phi^{-1}(ht)  }.
\end{align*}
Observe that $E_7\cap E_2 \subset E_8$. For $C_1$ large, one has $\PP(E_7)\le 2\rP\times\PP(E_8)=
2\PP\times\rP\times\rP_\rho(E_9)$  by Theorem \ref{RK:tran}. That $E_4\cap E_9\subset E_{10}$  
for $0<t<1$ and $\rP_\rho\times\PP(E_4) = e^{-2/u(0,0)}$ yields
\begin{align*}
\PP(E_7) \le 2e^{2/u(0,0)} \rP(E_{10}),  \quad 0<t<1.
\end{align*}
By Lemma \ref{lsc} and the fact that $\PP(T_0=0)=1$, there exists $r>0$ such that $h_x\ge 1/2$ for $|x|<r$. 
Hence,
\begin{align*}
\rP(E_{10}) \le \rP(\eta(x)\le 2\sqrt{2}C_1\sqrt{th}\sqrt{h}; |x|\le \phi^{-1}(ht))
\end{align*}
for all $t$ sufficiently small. Now the Gaussian lower tail estimates apply and the 
proof goes exactly as that of Lemma \ref{lem:all_x}.
$\square$
\end{proof}

It is worthy to point out that the generalized second Ray-Knight Theorem does not work 
well when we wish to study the asymptotic behavior of favorite points \emph{around infinity} 
for transient processes. Indeed, in the transient case,
the local time process at each fixed state is finite at infinity $L^x_\infty<\infty$ a.s. In particular, $L^0_\infty<\infty$,
hence its inverse $\tau(t)$ is (finite) constant for all $t$ sufficiently large, a.s.  It would be natural to consider the
local times under the true time scale $L^x_t$ rather than the time changed local time $L^x_{\tau(t)}$, but we have not been
able to do so.

On the other hand, we can consider another type of favorite sites at infinity of transient processes. Let $\mathcal  U_r
= \bra{y\in [-r,r]:  \sup_{|x|\le r} L^x_\infty = L^y_\infty}$ be  the set of favorite points within $[-r,r]$ at infinity and define
the favorite points process by $r\mapsto U_r = \inf\bra{|y|: y\in \mathcal U_r}$. We expect that Eisenbaum's Isomorphism
Theorem \cite[Thm. 8.1.1]{MarcusRosen06} plays the role of the second Ray-Knight Theorem. This is beyond the
scope of the present article, and will be studied in a separate paper.

\subsection{Symmetric L\'evy processes with a Gaussian component}

 In this subsection, we assume that $A\neq 0$ in \eqref{eq:LK_formula}. Denote $\psi_d(\la)= \psi(\la)-A^2\la^2$.
 It follows from \cite[Lem. 4.2.2]{MarcusRosen06} that
 \begin{align}\label{eq:p.24}
 \psi(\la) \asymp  \begin{cases}
 \la^2 & \la>1,\\
 \psi_d(\la) & 0<\la<1.
 \end{cases}
 \end{align}
Therefore, the local times exist and the process $X$ is transient or recurrent 
according to $\int_0^1 \frac{1}{\psi_d(\la)}\rd \la$  converges or diverges. 
It may be seen from the proof that the asymptotic behavior of the favorite 
points relies on the asymptotic  properties of $\sigma_0^2(x)$ defined in 
\eqref{F:sigma}, which is presented in the following lemma.

\begin{lemma}\label{l:Aneq0} Let $\sigma^2_0(x)$, $\wh\sigma^2_0(x)$, 
$\pi(\la)$, $\phi(x)$ be defined in Sections \ref{sec:intro} and \ref{sec:2}.
\begin{itemize}
 \item[(i)]  $\sigma^2_0(x)\asymp \phi(x)\asymp |x|$ as $|x|\to 0$.
 %\begin{align*}
%0<\liminf_{x\to 0}\frac{\sigma^2_0(x)}{|x|} \limsup_{x\to 0}\frac{\sigma^2_0(x)}{|x|}<\infty.
%\end{align*}
\item[(ii)] Under (C2), $\psi(\la)\asymp \pi(\la)$ as $\la \to 0$ and  $\wh\sigma^2_0(x)\asymp \phi(x)$ as $|x|\to \infty$.
 \end{itemize}
\end{lemma}
\begin{proof}
(i) holds by a Tauberian type result \cite[Thm 7.3.1]{MarcusRosen06},
%\footnote{\textcolor{blue}{Do we need to give a reference (e.g., Pitman (1969))?}}
and (ii) follows from \eqref{eq:p.24}. $\square$
\end{proof}

{Since the large time behavior of $V$ depends only on the asymptotic property 
of $\sigma_0^2(x)$ at infinity, we have that under Condition (C2) and recurrence,  
Part (ii) of Theorem \ref{theo} holds when $A \ne 0$.

The small time behavior of $V$ cannot be derived from the present proof when $A\neq 0$, 
the problem being that the lower tail estimate obtained through Cameron-Martin formula 
is not strong enough. In such case $\sigma_0^2(x)\asymp |x|$ around zero (the associated 
Gaussian process behaves locally like a Brownian motion),  one sees that our proof of 
Lemma \ref{lem:lowertail} gives an upper bound $ch^{\ga}$ for any $\ga<1$ but not for 
$\ga=1$.   The decay is not fast enough for the subsequent computations involving 
local times. This kind of phenomenon   appeared already in the  case of stable 
L\'evy processes as recalled in the following remark. }

%Apart from the asymptotic behavior of the characteristic exponent,  one notices other features brought by a Gaussian part to the study of favorite points.
\begin{remark}
{For $\al$-stable L\'evy processes, the Brownian case ($\al=2$) and the pure jump 
case ($1<\al<2$) may amount to different treatment regarding the lower tail estimate of the 
associated Gaussian process.
%We illustrate this point for the asymptotic behavior of $V$ at infinity.
Recall \cite[p.498]{MarcusRosen06} that the associated Gaussian process $\eta_\al$ for 
$\al$-stable process with $1<\al\le 2$ is a fractional Brownian motion with covariance 
$C(s,t) = \frac{C_\al}{2} (|s|^{\al-1} + |t|^{\al-1} - |s-t|^{\al-1})$. Note that  $\eta_2$ is a 
two-sided Brownian motion. The decay of the lower tail estimate obtained by Cameron-Martin 
formula is good enough for our purposes when $1<\al<2$, but is not sufficient for the 
subsequent computations when $\al=2$. Indeed, \cite[Lem. 11.5.1]{MarcusRosen06} 
shows that for all $1<\al\le 2$,  $\ep>0$,
\begin{align}\label{e:fbm_lowertail}
\rP \bigg(\sup_{|x|\le 1}{\eta_\al}(x)<\la \bigg)\le c \la^{\frac{2(1-\ep)}{\al}}.
\end{align}
 In the critical case $\al=2$,  we have the upper bound $c \la^\ga$ for any $\ga<1$, 
 a bound that is not sufficiently small for computations in the sequel, as pointed 
 out by Marcus and Rosen \cite[p.519]{MarcusRosen06}. However, the reflection 
 principle and the the independence of $\{\eta_2(x), x\ge 0\}$ and $\{\eta_2(x), 
 x\le 0\}$ actually implies
 \begin{equation}\label{e:bm_lowertail}
 \rP \bigg(\sup_{|x|\le 1}{\eta_2}(x)<\la \bigg)\le \sqrt{2	\over \pi} \la,
 \end{equation}
a better bound that is sufficient to derive results for $V$, see \cite[Sec. 11.2]{MarcusRosen06}.  }
\end{remark}

{In both references \cite{BassEisenbaumShi00,Marcus01}, the  authors viewed 
the associated Gaussian process as a time-changed Brownian motion and made use of 
\eqref{e:bm_lowertail} through Slepian's lemma. This approach remains valid for 
the small time behavior of the favorite points of a L\'evy process with both a 
Gaussian part and a pure jump part. Indeed, by replacing $\sigma_0^2(x)$ by $|x|$ 
back and forth  (and by the argument developed in Section  \ref{s:transient} 
if the process is transient),  we can remove the monotonicity condition of
 Marcus \cite{Marcus01} so that \cite[Thm 1.1]{Marcus01} holds as $t\to 0$ 
 when $\sigma_0^2(x)\asymp|x|$ around zero.  Summarizing the discussion, 
 we have proved the following theorem. Part (i) states that the small time 
 behavior of $V$ is the same as that of the Brownian part of $X$ without 
 any condition on the pure jump part. Part (ii) states that (under recurrence 
 condition) the large time behavior of $V$  is the same as that of the pure 
 jump part of $X$ subject to (C2).}

 \mk
%Let us now prove a version of Part (i) of Theorem \ref{theo} (small time asymptotic behavior of $V$).
%The associated
%Gaussian process (having variogram $\sigma_0^2(x)\asymp |x|$) behaves locally like a Brownian motion.  This is
%essentially covered by Marcus \cite[Thm. 1.1]{Marcus01}. Indeed, $\sigma_0^2(x)$ is asymptotically monotone, so
%one can apply Slepian's Lemma as in \cite{Marcus01}, but for small time, to prove lower tail estimates for the associated
%Gaussian process, see \cite[Lem. 2.4 and Lem. 3.3]{Marcus01} for more details.  We give, however, the following remark.
%
%\begin{remark}
%In their exposition of favorite points of stable processes, Marcus and Rosen pointed out in \cite[p.519]{MarcusRosen06}
%that the Gaussian lower tail estimates (for the corresponding fractional Brownian motion) obtained through Molchan's
%argument are not strong enough to derive asymptotic result for the favorite points of Brownian motion and a more specific
%method had to be applied.
%%which is somewhat the critical case.
%For the same reason, the present proof {\textcolor{blue}{of Theorem \ref{theo} }} cannot cover the cases $\ov\al=2$
%and $\ov\be=2$. It is, however, possible to treat these two cases by combing
%the results of Section 2 and the approach of Marcus \cite{Marcus01}.\footnote{\textcolor{blue}{In the cases of $\ov\al=2$
%and $\ov\be=2$, does Theorem \ref{theo} still hold?}}
%\end{remark}

%To conclude, we state the following result.\footnote{\textcolor{blue}{Does this theorem follow from the above remark?
%Do you still exclude the case of $\ov\al=2$ and $\ov\be=2$?}}

\begin{theorem} 
Let $A\neq 0$ in \eqref{eq:LK_formula}.
\begin{itemize}
\item[(i)] For $a>8$,
\begin{align*}
 \lim_{t\to 0}  \frac{ V_t}{ \frac{L^0_t}{(\log L^0_t)^a}}=\infty \quad a.s.
\end{align*}
\item[(ii)] Under Condition (C2) and $\int_0^1\frac{1}{\psi_d(\la)}\rd\la=\infty$, for all $a>2(\ov\be-1)/(2-\ov\be)$,
\begin{align*}
 \lim_{t\to \infty}  \frac{ V_t}{\phi^{-1}\left( \frac{L^0_t}{(\log L^0_t)^a}\right)}=\infty \quad a.s.
\end{align*}
\end{itemize}
\end{theorem}

{
We finish the paper with an open problem:  can one describe the small (resp. large) 
time behavior of $V$ when the L\'evy measure satisfies (C1) with $\ov\al=2$ (resp. 
(C2) with $\ov\be=2$) and $X$ is a pure jump process? One may still obtain lower 
tail estimates with the approach of \cite{BassEisenbaumShi00,Marcus01}. However,
\eqref{eq:ratio_contr_1}-\eqref{eq:equiv_1} fail to hold and it is not clear to 
us how to find a monotone function
which is equivalent to $\sigma_0^2(x)$. }

\section*{Acknowledgements}
The research of Y. Xiao is partially supported by NSF grants DMS-1612885 and DMS-1607089. 
This paper was completed while Y. Xiao was visiting the Mittag-Leffler Institute in Sweden. 
He thanks the staff members for providing him with an excellent working environment.

%%%%%%%%%%%%%%%%%%%%%%%%%%%%%%%%%%%%%%%%%%%%%%%%%%%%%%%%%%%%%%%%%%%%%%%%%%%%%%%%
%%%%%%%%%%%%%%%%%%%%%%%%%%%%%%%%%%%%%%%%%%%%%%%%%%%%%%%%%%%%%%%%%%%%%%%%%%%%%%%%
%---------------------------------------------------------------------------
% Bibliographie
%---------------------------------------------------------------------------
\bibliographystyle{plain}
\bibliography{xyangbiblio}

\end{document}